\documentclass[final]{siamart0516}

\usepackage{lipsum}
\usepackage{amsfonts}
\usepackage{graphicx}
\usepackage{epstopdf}
\usepackage{algorithmic}
\usepackage{amsmath,amssymb}
\usepackage{epstopdf}
\ifpdf
  \DeclareGraphicsExtensions{.eps,.pdf,.png,.jpg}
\else
  \DeclareGraphicsExtensions{.eps}
\fi

\renewcommand{\theenumi}{\roman{enumi}}

\newcommand{\TheTitle}{A Quadratic-Time Algorithm for General Multivariate Polynomial Interpolation} 
\newcommand{\TheAuthors}{M.~Hecht, B.L.~Cheeseman, K.B.~Hoffmann and I.F.~Sbalzarini}

\headers{General Polynomial Interpolation}{\TheAuthors}

\title{{\TheTitle}\thanks{Submitted to the editors DATE.
\funding{This work was funded by the Max Planck Society, and by the German Federal Ministry of Education and Research (BMBF) under funding code 031L0044.
B.L.C. further acknowledges financial support though a DIGS-BB fellowship, awarded by the DFG-funded Excellence Graduate School of TU Dresden under code DFG-GSC-97.}}}

\author{
  Michael Hecht\footnotemark[1] \thanks{MOSAIC Group, Chair of Scientific Computing for Systems Biology, Faculty of Computer Science, TU Dresden
\& Center for Systems Biology Dresden, Max Planck Institute of Molecular Cell Biology and Genetics, Dresden\
    (\email{hecht@mpi-cbg.de}, \email{cheesema@mpi-cbg.de}, \email{karlhoff@mpi-cbg.de}, \email{ivos@mpi-cbg.de}, \url{http://mosaic.mpi-cbg.de}).}
  \and
  Bevan L.~Cheeseman\footnotemark[1]
  \and
  Karl B.~Hoffmann\footnotemark[1] 
  \and   
   Ivo F.~Sbalzarini\footnotemark[1] 
}

\usepackage{amsopn}

\ifpdf
\hypersetup{
  pdftitle={Supplementary Materials: \TheTitle},
  pdfauthor={\TheAuthors}
}
\fi

\newcommand{\R}{\mathbb{R}}
\newcommand{\N}{\mathbb{N}}
\newcommand{\Z}{\mathbb{Z}}

\newcommand{\PP}{\mathbb{P}}

\newcommand{\Vc}{\mathcal{V}}

\newcommand{\Ec}{\mathcal{E}}

\newcommand{\Pc}{\mathcal{P}}

\newcommand{\Ic}{\mathcal{I}}

\newcommand{\Oc}{\mathcal{O}}
\newcommand{\ee}{\varepsilon}
\newcommand{\lo}{\longrightarrow}
\newcommand{\li}{\left}
\newcommand{\re}{\right}
\newcommand{\mi}{\,\,\big|\,\,}

\newcommand{\spann}{\mathrm{span}}

\newtheorem{experiment}{Experiment}
\begin{document}

\maketitle

\begin{abstract}
For $m,n \in \N$, $m\geq 1$ and a given function 
$f : \R^m\lo \R$ the \emph{polynomial interpolation problem} (PIP) is to determine a \emph{generic node set} $P \subseteq \R^m$ and the coefficients of the uniquely defined polynomial $Q\in\R[x_1,\dots,x_m]$ in $m$ 
variables of degree $\deg(Q)\leq n \in \N$ that fits $f$ on $P$, i.e., $Q(p) = f(p)$, $\forall\, p \in P$.
We here show that in general, i.e., for arbitrary $m,n \in \N$, $m \geq 1$, 
there exists an algorithm that determines $P$ and computes the $N(\mbox{m,n})=\#P$ coefficients of $Q$ in $\Oc\big(N(\mbox{m,n})^2\big)$ time using $\Oc\big(\mbox{m}N(\mbox{m,n})\big)$ storage, without inverting the occurring \emph{Vandermonde matrix}. We provide such an algorithm, termed PIP-SOLVER, based on a recursive decomposition of the problem and prove its correctness. 
 Since the present approach solves the PIP without matrix inversion, it is computationally more efficient and numerically more robust than previous approaches. 
 We demonstrate this in numerical experiments and compare with previous approaches based on matrix inversion and linear systems solving. 
\end{abstract}


\section{Introduction}
In scientific computing, the problem of interpolating a function $f : \R^m \lo \R$, $m \in \N$ is ubiquitous. 
Because of their simple differentiation and integration, as well as their pleasant vector space structure, polynomials $Q \in \R[x_1,\dots,x_m]$ in $m$ variables of degree $\deg(Q) \leq n$, 
$m,n \in \N$, are a standard choice as interpolants and are fundamental in ordinary differential equation (ODE) and partial differential equation (PDE) solvers. 
For an overview we refer the reader for instance to \cite{NA2} and  \cite{Na}.  
Thus, the \emph{polynomial interpolation problem} (PIP) is one of the most fundamental numerical problems. We note that a polynomial $Q \in \R[x_1,\dots,x_m]$ in $m$ 
variables of degree $\deg(Q)=n$ possesses $N(m,n):={m+n \choose m}$ monomials and formulate the PIP as follows:

\begin{problem}[PIP]\label{PolyInt}
Given parameters $n,m \in \N$ and a function $f : \R^m \lo \R$. Then the problem is to choose $N(m,n)$ generic nodes 
$P=\{p_1,\dots,p_{N(m,n)}\} \subseteq \R^m$ such that 
there is exactly one polynomial 
$Q \in \R[x_1,\dots,x_m]$ of degree $\deg(Q)\leq n$ that coincides with $f$ on $P$, i.e., $Q(p)=f(p)$ for all $p\in P$, and to determine $Q$ once $P$ has been chosen.
\end{problem}

The function $f : \R^m \lo \R$ is given in the sense that for all $x \in \R^m$ the value of $f(x)$ can be evaluated in $\Oc(1)$ time. Note that if $f$ is a polynomial of degree $\deg(f)\leq n$ then $Q=f$. 
The $1$-dimensional case (\emph{Newton or Lagrange Interpolation}) can be solved by various algorithms for which error estimates and numerical stability are well-known \cite{LIP}. 
However, many data sets in scientific computing are functions of more than one variable and therefore require multivariate polynomial interpolation. 
From the classical closed-form expressions for the $1$-dimensional case, it is known that accuracy and numerical stability of the PIP solution depend on the position of the nodes $P$. 
Specific adequate arrangements of nodes, similar to \emph{Chebyshev nodes} in dimension $1$ \cite{burden}, are also known for the $2$-dimensional square \cite{Erb} and for the 
$d$-dimensional cube \cite{Bos}. An excellent survey of approaches to the general problem is given in \cite{Gasca2000}. 
However, the regularity condition of the general PIP has never been characterized without the strong assumption that 
the nodes are distributed along previously fixed lattices or grids \cite{Bos,Erb,Gasca2000,Chung}.
For dimensions $m>3$, general characterizations become complex and were so far considered infeasible \cite{Gasca2000}. 

Note that due to the famous Theorem of Sard, which was later generalized by Smale \cite{smale}, the set $\Pc_{m,n}$ of all generic node sets with respect to the parameters 
$m,n \in \N$ is a set of second category in the sense of Baire and, therefore, its compliment
$\Pc_{m,n}^C \subseteq \oplus_{k=1}^{N(m,n)}\R^m$  
is a  zero set with respect to the Lebesgue measure. In other words,  
an arbitrary perturbation of a degenerate set $P_0$ will result in a generic set $P$ with probability $1$. However, 
in scientific applications we have to require in addition that Problem \ref{PolyInt} can be solved numerically stably and accurately with respect to $P$.
Even if generic nodes $P$ are given, classical approaches require the inversion of the \emph{multivariate Vandermonde matrix} $V_{m,n} \in \R^{N(m,n)\times N(m,n)}$ in order to compute the coefficients of $Q$.
Therefore, they are limited by the cost of \emph{matrix inversion}. The fastest known algorithm for matrix inversion is the \emph{Coppersmith-Winograd algorithm} \cite{COPPER}, which requires runtime in $\Oc(N(m,n)^{2.3728639})$ in its most efficient version \cite{FAST}. 
However, the \emph{Coppersmith-Winograd algorithm} is rarely used in practice, because it only breaks even for matrices so large that memory problems become prevalent on modern hardware \cite{robinson}.
The algorithm that is mostly used in practice is the \emph{Strassen algorithm} \cite{strassen}, which runs in $\Oc(N(m,n)^{2.807355})$. 
Alternatively,  one can solve the corresponding system of linear equations by \emph{Gaussian elimination} in $\Oc(N(m,n)^3)$. All of these approaches require $\Oc(N(m,n)^2)$ storage to hold the Vandermonde matrix. 
Moreover, the condition number of $V_{m,n}$ limits the numerical robustness and accuracy with which these approaches can solve the PIP. Numerically, classical approaches thus become intractable with increasing $N(m,n)$.

In contrast to those classical approaches, we here show that it is possible to recursively decompose the problem into sub-problems of lesser degree or lesser dimension, until we reach only linear or
$1$-dimensional sub-problems that can efficiently and robustly be solved, {\em independent} of the condition number of the Vandermonde matrix.
Consequently, we solve the PIP numerically accurately and efficiently without inverting the Vandermonde matrix. The resulting algorithm, PIP-SOLVER, has a runtime in $\Oc(N(m,n)^2)$ and requires storage in $\Oc(mN(m,n))$. 
Both time and space complexity are therefore lower than those of previous approaches, and we avoid the numerical inaccuracies and instabilities that occur during matrix inversion.

\subsection{Outline}  We first present the main results of this article in section \ref{main}. Afterwards, 
in section \ref{FormVander}, we formally introduce general \emph{Vandermonde systems} and deduce a notion of generic node sets $P$, implying that the nodes are not allowed to lie on 
a hypersurface of degree $\leq n$, where $n = \deg(Q)$ is the degree of the polynomial one wants to determine. Furthermore, we assert how the 
Vandermonde systems can be used to solve Problem \ref{PolyInt} over generic node sets. In section \ref{Newt} we consider the special $1$-dimensional case of the PIP also known as 
\emph{Newton or Lagrange Interpolation} and review  
 classical results. Moreover, we study the special linear case $m \in \N,n=1$ in section \ref{linear}. In section \ref{ALGCurv} we introduce some basic concepts of the theory of algebraic curves and state 
B\'{e}zout's Theorem, which relates the number of intersections of algebraic curves to their degrees. The key result is that two algebraic curves $V$, $W$ without common components intersect in at most 
$\deg(V)\deg(W)$ points, counted with multiplicity. In section \ref{2D} we use this fact to determine configurations of nodes whose intersections with specific algebraic curves
guarantees genericity in dimension $m=2$. In section \ref{Decomp} we prove that the general PIP can be split into two sub-problems and in section \ref{Experiments} we use this splitting  
to formulate a PIP-SOLVER, proving the main theorem. Moreover, in section \ref{EX} we  
present several numerical experiments that illustrate our results. Afterwards, in section \ref{APP},  we mention non-obvious applications and finally, in section \ref{Conc}, we discuss remaining problems and possible generalizations of our approach.

\section{Main Result}\label{main}

The main result of this article is that both problems, namely finding a generic set of nodes $P$ and computing the interpolating polynomial $Q$, can be split into sub-problems
with parameters \mbox{$(m-1,n)$} and \mbox{$(m,n-1)$}. Consequently, we can  recursively decompose the original PIP to \emph{linear} or \emph{1-dimensional} sub-PIPs. 
The combination of the solutions of these sub-problems then yields a solution of the original PIP. Since the linear and $1$-dimensional sub-PIPs can be solved in $\Oc(m^2)$ and $\Oc(n^2)$, respectively, we obtain:

 \begin{theorem}[Main Result] Let $m,n \in \N$ and $f : \R^m \lo \R$ be a given function. Then there exists an algorithm with runtime complexity $\Oc\big(N(\mbox{m,n})^2\big)$ requiring storage $\Oc(mN(\mbox{m,n}))$ 
 to compute a generic set $P \subseteq \R^m$ and a polynomial $Q \in \R[x_1,\dots,x_m]$ with $\deg (Q) \leq n$,  such that $Q$ is the unique solution of the PIP with respect to $P$ and $f$.
 \end{theorem}

Note that the theoretical lower bound for \emph {inversion} of an $N\times N$ matrix is given by \emph{matrix multiplication}, which has complexity at least $\Oc(N^2\log(N))$ \cite{tveit}. 
Thus, the statement above implies a possibility of overcoming this barrier by avoiding the inversion. 

\section{Systems of Polynomials} \label{FormVander}

We follow  \cite{MultiVander}  to introduce some basic notions and helpful notations. 
While a polynomial $Q \in \R[x_1,\dots,x_m]$ of degree $\deg(Q)=n$ 
possesses $N(m,n):={m+n \choose m}$ monomials, the number of monomials of degree $k$ is given by  $M(m,k):={m+k \choose m} - {m+k-1 \choose m}$.
We  enumerate the coefficients $c_0,\dots,c_{N(m,n)}$ of 
$Q$ such that 
\begin{align}
 Q(x) & =  c_0 + c_1 x_1  + \cdots + c_mx_{m}+ c_{m+1}x_1^2+c_{m+2}x_1x_2 + \cdots  + c_{2m}x_1x_m \nonumber \\ 
 & + c_{2m+1}x_2^2+ \cdots  +  c_{M(m,n-1)+1}x_1^n   + c_{M(m,n-1)+2}x_1^{n-1}x_2 +\cdots\nonumber \\ 
 & +  c_{N(m,n)-1}x_{m-1}x_m^{n-1} + c_{N(m,n)-1}x_m^n\,. \label{C}
\end{align}
 We assume that $0 \in \N$ and denote by $I=(i_1,\dots,i_m) \in \N^m$ 
a multi-index of order  $\#I:= \sum_{k=1}^mi_k$. The multi-index is used to address the monomials of a multivariate polynomial. 
For a vector $x=(x_1,\dots,x_m)$ and $I \in \N^m$ we define 
\begin{equation*}
\PP_I(x):=  x_1^{i_1}\cdots x_m^{i_m}
\end{equation*}
and by ordering the $M(m,k)$, we define multi-indices $I \in \N^m$ of order  $\#I =k$ with respect to lexicographical order, i.e., 
$I_1 =(k,0,\dots,0)$, $I_2=(k-1,1,0\dots,0)$,$\dots$, $I_{M(m,k)} =(0,\dots,0,k)$. The $k$-th \emph{symmetric power} 
$x^{\odot k }=(x_1^{\odot k },\dots,x_{M(m,k)}^{\odot k }) \in \R^{M(m,k)}$ is defined  by the entries 
\begin{equation}
x^{\odot k }_i:= \PP_{I_i}(x) \,, \quad i=1,\dots,M(m,k)\,.
\end{equation}
Thus, $x^{\odot 0 }_i =1$, $x^{\odot 1 }_i =x_i$.  
\begin{definition} Given parameters $n,m \in \N$. For a set $P=\{p_1,\dots,p_{N(m,n)}\} \subseteq \R^m$ of nodes, 
with $p_i = (p_{1,i}, \dots,p_{m,i})$, 
we define the \emph{multivariate Vandermonde matrix} $V_{m,n}(P)$ by 
$$ V_{m,n}(P) = \li(\begin{array}{ccccc}
                       1   & p_1   & p_1^{\odot 2}  & \cdots      & p_1^{\odot n}  \\
                       1  & p_2 &p_2^{\odot 2} & \cdots & p_2^{\odot n}\\
                       1 & p_3 & p_3^{\odot 2}& \cdots & p_3^{\odot n}\\
                       1 & \vdots &\vdots & \ddots& \vdots \\
                       1 & p_{N(m,n)}  & p_{N(m,n)}^{\odot 2} & \cdots & p_{N(m,n)}^{\odot n}\\
                      \end{array}\re)
                      \,.$$ 
                      \end{definition}
We call a set $P =\{p_1,\dots,p_{N(m,n)}\} \subseteq \R^m$ of nodes \emph{generic} if and only if the  Vandermonde matrix $V_{m,n}(P)$
is regular. Thus, the set of all generic node sets is given by 
$\Pc_{m,n}=\li\{ P \subseteq \R^m \mi \det\big(V_{m,n}(P)\big)\not = 0\re\}$, which is open in $\R^m$ since $P \mapsto \det\big(V_{m,n}(P)\big)$ is a continuous function. Given a real-valued function $f$ on $\R^m$ and assuming that there exist generic nodes $P=\{p_1,\dots p_{N(m,n)}\} \subseteq \R^m$, 
then the linear system of equations $V_{m,n}(P)x = F$, with
$$  x=\big(x_1,\dots,x_{N(m,n)}\big)^T\,\,\,\text{and}\,\,\, F=\big(f(p_1),\dots,f(p_{N(m,n)})\big)^T $$ 
possesses the unique solution $s=(s_1,\dots,s_{N(m,n)})\in \R^{N(m,n)}$ given by
$$ s = V_{m,n}(P)^{-1}F \,.$$ 
Thus, by setting $c_i:=s_i$ for the coefficients of $Q\in\R[x_1,\dots,x_m]$, enumerated as in \eqref{C}, we have uniquely determined the solution of Problem \ref{PolyInt}.
The essential difficulty herein lies in finding a good generic node set by solving the following:
\begin{problem}
Given parameters $n,m \in \N$, choose $N(m,n)$ generic nodes 
$P=\{p_1,\dots,p_{N(m,n)}\} \subseteq \R^m$ such that inversion of the Vandermonde matrix  $V_{m,n}(P)$ is numerically stable and accurate.
\label{VV}
\end{problem}
Moreover, our observations prove the following well-known result first mentioned in  \cite{Chung} and again in \cite{MultiVander}: 

\begin{theorem}\label{generic} Let $m,n \in \N$  and $P \subseteq \R^m$, $\# P = N(m,n)$. The Vandermonde matrix $V_{m,n}(P)$
is regular if and only if 
the  nodes $P$  
do not belong to a common algebraic hypersurface of
degree $\leq n$, i.e., if there exists no polynomial $Q \in \R[x_1,\dots,x_m]$, $Q \not =0$ of degree $\deg(Q) \leq n$  such that $ Q(p) = 0 $ for all $p \in P$.  
\end{theorem}

\begin{remark} In other words: Theorem \ref{generic} says that $P$ is genric if and only if
the homogeneous Vandermonde problem 
$$V_{m,n}(P)x = 0$$
possesses no non-trivial solution, which is equivalent to the fact that there is no hypersurface $V$ of degree $\deg(V)\leq n$ with $P \subseteq V$.  
\end{remark}

\section{Special Cases}\label{SC}
We solve the general PIP by recursively decomposing it into linear or $1$-dimensional sub-PIPs. How to solve these sub-PIPs is reviewed in this section.
To this end, we define the following central concepts:

\begin{definition}
Let $\tau : \R^m \lo \R^m$ be given by $\tau(x) = Ax +b $, where $A \in \R^{m\times m}$ is a full-rank matrix, i.e. $\mathrm{rank}(A) =m$, and $b \in \R^m$. Then we call 
$\tau$ an \emph{affine transformation} on $\R^m$. 
\end{definition}

\begin{definition}\label{proj}
For every ordered tuple of integers $i_1,\dots,i_k \in \N$, $i_{q} < i_p$ if $1 \leq q <p \leq k$, we consider  
$$H_{i_1,\dots, i_k}=\li\{(x_1,\dots,x_n) \in \R^m \mi x_j = 0 \,\, \text{if}\,\, j \not \in \{i_1,\dots,i_k\}\re\}$$ the $k$-dimensional hyperplanes spanned by the 
$i_1,\dots,i_k$-th coordinates. We  denote by $ \pi_{i_1,\dots, i_k} : \R^m \lo H_{i_1,\dots,i_k}$ and $i_{i_1,\dots, i_k} : \R^k \hookrightarrow  \R^m$, with $i_{i_1,\dots, i_k}(\R^k)=H_{i_1,\dots,i_k}$
the natural projections and embeddings. 
We denote by
$\pi_{i_1,\dots,i_k}^* : \R[x_1,\dots,x_m] \lo \R[x_{i_1},\dots x_{i_k}]$, $i^*_{i_1,\dots, i_k} : \R[y_{1},\dots y_{k}] \hookrightarrow \R[x_1,\dots,x_m]$ the induced projections and embeddings 
on the polynomial ring. 
\end{definition}

\begin{definition}\label{Hxi} Let $m,n \in \mathbb{N}$, 
$\xi_1,\dots,\xi_m \in \R^m$ an orthonormal frame (i.e., $\li<\xi_i,\xi_j\re>=\delta_{ij}$, $ \forall 1\leq i,j\leq m$, where $\delta_{ij}$ denotes the Kronecker symbol), and $b \in \R^m$. 
For $\Ic \subseteq \{1,\dots,m\}$
we consider the hyperplane 
$$H_{\Ic,\xi,b}:= \li\{x \in \R^m \mi \li<x-b,\xi_i\re>=0\,, \forall \, i \in \Ic \re\}\,.$$ 
Given a function $f : \R^m \lo \R$ we say that a set of nodes $P \subseteq H_\Ic$ and a polynomial $Q \in \R[x_1,\dots,x_m]$ with $\deg(Q) \leq n$ 
\emph{solve the PIP} with respect to $n,k=\dim H_\Ic$ and $f$ on $H_\Ic$ if and only if $Q(p) =f(p)$ for all $p \in P$ and whenever there is  a $Q'\in \R[x_1,\dots,x_m]$ with $\deg(Q')\leq n $ and $Q'(p)=f(p)$
for all $p \in P$ then $Q'(x)=Q(x)$ for all $x \in H_\Ic$.
\end{definition}

\begin{definition}Let  $H \subseteq \R^m$ be a hyperplane of dimension $k\in \N$ and $\tau : \R^m\lo \R^m$ an affine transformation such that 
$\tau(H) =H_{1,\dots,k}$.
Then we denote by
$$\tau^*: \R[x_1,\dots,x_m] \lo \R[x_1,\dots,x_m]$$
the \emph{induced transformation} on the polynomial ring defined over the monomials as:
$$ \tau^*(x_i) = \eta_1x_1+\cdots + \eta_mx_{m}\quad \text{with} \quad \eta=(\eta_1,\dots,\eta_m) = \tau(e_i)\,,$$
where $e_1,\dots,e_m$ denotes the standard basis of $\R^m$.
\end{definition}

\begin{lemma}\label{TT} Let $m,n \in \N$ and $P$ be a generic node set with respect to $m,n$. Further let $\tau : \R^m \lo \R^m$, $\tau(x) =Ax +b$ be an affine transformation. 
Then, $\tau(P)$ is also a generic set with respect to $m,n$.
\end{lemma}
\begin{proof} Assume there is  a polynomial $Q_0 \in \R[x_1,\dots,x_m]$ with $\deg (Q_0) \leq n$ such that $Q_0(\tau(P))=0$. Then setting $Q_1:=\tau^*(Q_0)$ yields a non-zero polynomial with $\deg(Q_1)\leq n$ and
$$Q_1(P) = \tau^*(Q_0)(P)= Q_0(\tau(P))=0\,, $$
which contradits that $P$ is generic. Hence, $\tau(P)$ must be generic.
\end{proof}

\subsection{1-dimensional Interpolation}\label{Newt}
We review the classical results in the special case of dimension $m=1$ and refer the reader to \cite{LIP} for a more detailed discussion of $1$-dimensional interpolation. 
In one dimension, the Vandermonde matrix $V_{1,n}(P)$ takes its classical form 
$$ V_{1,n}(P) = \li(\begin{array}{cccc}
                       1   & p_1    & \cdots      & p_1^{ n}  \\
                       \vdots & \vdots & \ddots& \vdots \\
                       1 & p_{n+1}   & \cdots & p_{n+1}^{n}\\
                      \end{array}\re)
                      \,.$$
For the matrix $V_{1,n}(P)$ to be regular, the nodes $p_1,\dots,p_{n+1}$ have to be pairwise different. Theorem \ref{generic} implies that this is also a sufficient condition for the nodes $P$ to be generic.
In light of this fact, we consider the \emph{Lagrange polynomials} 
$$L_j(x) = \prod_{h=1, h\not=i}^{n+1} \frac{x-p_h}{p_j-p_h} \,, \quad j =1,\dots,n$$ fulfilling  $L_{j}(p_k) =\delta_{jk}$, where $\delta_{jk}$ is the Kronecker symbol. 
Note that $V_{1,n}$ induces a linear bijection  
$\varphi : \R^{n+1} \lo \R_n[x]$, where $\R_n[x]$ is the vector space of real polynomials $Q \in \R[x]$ with $\deg(Q) \leq n$ and $\varphi(Q)$ is given by the polynomial with coefficients $V_{1,n}s$.  
Hence, if we use the $L_j$ instead of $\{1,x,x^2,\dots,x^n\}$ as a basis of $\R_n[x]$, the transformed Vandermonde matrix 
becomes the identity matrix. Thus, the solution of the PIP in dimension $m=1$ is given by choosing pairwise different nodes $P$ and setting
\begin{equation*}
Q(x) = \sum_{j=1}^{n+1}f(p_j)L_j(x)\,.
\end{equation*}
The coefficients of $Q$ are given by:
\begin{equation}\label{Form}
 c_i = Q^{(i)}(0)= \sum_{j=1}^{n+1} f(p_j)L_j^{(i)}(0)\,, \quad i =0,\dots,n \,,
 \end{equation}
where $Q^{(i)},L_i^{(i)}$ denotes the $i$-th derivative of $Q$ and $L_i$, respectively. Thus, starting from an analytical expression for $L_j^{(i)}(0)$ we can establish an analytical formula 
for the $c_i$ that requires computing $n$ derivatives of 
$n+1$ polynomials, therefore incurring a computation cost of $\Oc(n^2)$ overall. Consequently, the $1$-dimensional PIP can be solved as follows:

\begin{proposition}\label{1D} Let $m,n,j \in \mathbb{N}$, $1\leq j\leq m$, $\xi_1,\dots,\xi_m \in \R^m$ an orthonormal frame, $b \in \R^m$, and  $\Ic =\{1,\dots,m\}\setminus\{j\} $ such that $H_\Ic:=H_{\Ic,\xi,b}$ is a 
line. Further let $f : \R^m \lo \R$ be a computable function. 
\begin{enumerate} 
 \item[i)] If $\xi_1,\dots,\xi_m $ is the standard basis and $b_j =0$, then there exists an algorithm that solves the PIP with respect to $(1,n)$ and $f$ on $H_\Ic$ in $\Oc(n^2)$. 
 \item[ii)] If $\xi_1,\dots,\xi_m $ is an arbitrary frame, then there exists an algorithm that solves the PIP with respect to $(1,n)$ and $f$ on $H_\Ic$ in $\Oc(mnN(m,n)+n^2)$.  
\end{enumerate}
\label{PIP1D}
\end{proposition}

\begin{proof}For both cases let $\{j\}= \{1,\dots,m\}\setminus \Ic $. Then we define $t=t(x): \R^m \lo \R$ by $t(x) =\li<x-b,\xi_j\re>$. Let $p_1,\dots,p_{n+1} \subseteq H_\Ic$ be pairwise different, then 
$q_i:= t(p_i)$, $i=1,\dots,n+1$ are also pairwise different. Setting $\hat f(q_i) = f(p_i)$, the reasoning above 
yields a uniquely determined polynomial $\hat Q(t) = \hat c_0+ \hat c_1 t + \cdots + \hat c_n t^n \in \R[t]$ with $\deg(\hat Q)\leq n$ such that $\hat Q(q_i)=\hat f(q_i)$, $i=1,\dots,n+1$. 
Now setting  
\begin{equation*}
  Q(x) =   \hat c_0 + \hat c_1t(x) + \cdots \hat c_n t(x)^n \in \R[x_1,\dots,x_m]
\end{equation*}
yields a polynomial with $Q(p_i) = \hat f(q_i) = f(p_i)$, $i=1,\dots,m$. Since the restriction of $t$ to $H_\Ic$ becomes a bijection and $\hat Q$ is uniquely determined, also $Q$ is uniquely 
determined on $H_\Ic$.
Thus, $P=\{p_1,\dots, p_{n+1}\}$ and $Q$ solve the PIP on $H_\Ic$ with respect to $(1,n)$ and $f$. 

If $\xi_1,\dots,\xi_m $ is the standard basis we have $t(x)= \li<x-b,\xi_j\re>=(x-b)_j =x_j-b_j$. 
Hence 
$$ \hat c_i t(x)^i = \hat c_i\sum_{k=0}\binom{i}{k}(-1)^{i-k}x_j^kb_j^{i-k}$$
and therefore $Q(x) = c_0 + c_1x_j + c_2 x_j ^2 + \cdots c_nx_j^n$, where the coefficients $c_i$ of $Q$ are given by: 
\begin{equation}\label{1Dst}
  c_i = \sum_{q=i}^{n}  (-1)^{q-i}\binom{q}{i}\hat c_qb_j^{q-i}\,.
\end{equation}
Under the assumption that floating-point arithmetic operations are $\Oc(1)$, for instance in \cite{Preiss:1998} it is shown that $\binom{n}{k},n,k \in \N$ can be computed by dynamic programming in $\Oc(nk)$. 
Hence, due to Eq.~\eqref{Form}, the coefficients $\hat c_0, \dots, \hat c_n$ of $\hat Q$ can be determined in $\Oc(n^3)$ for every $i=1,\dots,n$ and therefore in $\Oc(n^4)$ for all $i =1,\dots,n$. 
In the special case where  $b_j =0$, we have $c_i = \hat c_i$, 
which determines the coefficients of $Q$ in $\Oc(n^2)$, proving $(i)$. 

If $\xi_1,\dots,\xi_m $ is an arbitrary frame, let $c_{J,K}$, $J=\{j_1,\dots,j_l\}$, $K=\{k_1,\dots,k_l\}$, $|K|:= \sum_{h=1}^lk_h = i$
be the coefficients of $Q$ with respect to the monomials $x_{j_1}^{k_1} x_{j_2}^{k_2} \cdots  x_{j_l}^{k_l}$. 
They are given by:  
\begin{equation}\label{Dxi} 
 c_{J,K} = D^{i-1}_{e_{J,K}}Q(b)=  i!\cdot \hat c_i\cdot \prod_{h=1}^l\li<e_{j_h},\xi_i\re>^{k_h} =  i!\cdot \hat c_i\cdot \prod_{h=1}^l\xi_{j,j_h}^{k_h}  \,,
\end{equation}
where $D^{i-1}_{e_{J,K}}Q(x)$ denotes the \mbox{$(i-1)$}-th derivative of $Q$ at $x$ in standard coordinate directions $e_{j_1}, \dots, e_{j_l}$ with respect to multiplicities $k_1,\dots,k_l$. $\xi_{j,j_h}$ denotes the $j_h$-th element of $\xi_j$. 
Thus, in addition to the $\Oc(n^2)$ operations required to determine $ \hat c_0, \dots, \hat c_n$  
we need $\Oc(mnN(m,n))$ operations to determine \eqref{Dxi} for every coefficient of $Q$, which shows $(ii)$.
\end{proof}

\begin{corollary} Let $m,n \in \mathbb{N}$, $\xi_1,\dots,\xi_m \in \R^m$ an orthonormal frame, $b \in \R^m$, and $\Ic \subseteq \{1,\dots,m\}$ such that $H_\Ic:=H_{\Ic,\xi,b}$ is a 
line. Further let $f : \R^m \lo \R$ be a computable function. Then, there exists a linear transformation $ \varphi : \R^m \lo \R^m$, with $\varphi(H_\Ic)= \spann(e_1)+b$,  and an algorithm to determine a generic set of nodes $P \subseteq H_\Ic$ and a polynomial $Q(x) \in \R[x_1,\dots,x_m]$ in $\Oc(n^2)$ such that
\begin{enumerate} 
 \item[i)] $\varphi^*(Q)$ (i.e., the induced transformation on polynomials) solves the PIP on $H_\Ic$ with respect to $P$ and $f$. 
 \item[ii)] $Q$ solves the PIP on $\varphi(H_\Ic)$ with respect to $P$ and $f \circ \varphi^{-1}$.
\end{enumerate}

\end{corollary}

\begin{proof} We set 
$$A:= \li(\begin{array}{ccc}
                       \vdots   &  &\vdots       \\
                       \xi_{1}  & \cdots & \xi_{m} \\
                      \vdots   &  &\vdots      
                      \end{array}\re) \in \R^{ m\times m}
                      \,
                      $$
                      and $\varphi(x) = A^Tx$. Then $\varphi(H_\Ic)=\spann(e_1)+A^Tb$. Thus, due to Proposition \ref{1D}, 
                      we can determine a generic set of nodes $\hat P \subseteq \varphi(H_\Ic)$ and a polynomial $Q$ in $\Oc(n^2)$ such that $Q$ solves the PIP with respect to $\hat P$ and $f \circ \varphi^{-1}$, proving $(ii)$. 
                      Setting $P = \varphi^{-1}(\hat P)$ proves $(i)$.  
                      \end{proof}

\begin{remark}
Alternatively to our formula \eqref{Form}, an explicit $\Oc(n^2)$ approach based on $LU$-decomposition of $V_{1,n}(P)$ can be used  \cite{NASA}. This also allows computing derivatives and integrals efficiently and accurately once 
the decomposition has been obtained. In any case, choosing \emph{Chebyshev nodes} $x_k =\lambda\cos\li(\frac{2k-1}{2n}\pi\re)$, $k=1,\dots,n+1$, $\lambda \in \R^+$ minimizes the approximation error 
of $1$-dimensional interpolation, see \cite{Stewart}.

\end{remark}


\subsection{Linear Interpolation}\label{linear}
In the linear case, i.e., if  $n=1$, we have that for $p_1,p_2\in\R^m$, $p_1 \not =p_2$, there exists a uniquely determined line $L(p_1,p_2)$ containing $p_1$ and $p_2$. Choosing any point $p_3 \in \R^m$ with 
$p_3 \not \in L(p_1,p_2)$ uniquely determines a plane $H_2(p_1,p_2,p_3)$ containing $p_1,p_2,p_3$. Iterating this process and setting $H_1(p_1,p_2)=L(p_1,p_2)$ yields an efficiently computable procedure for constructing $2 \leq k+1 \leq m+1$ nodes $p_1,\dots,p_{k+1}$ that
belong to exactly one $k$-dimensional hyperplane $H_k(p_1,\dots,p_{k+1})\subseteq \R^m$. In this case, we call the nodes $p_1,\dots,p_{k+1} \in \R^m$  
\emph{linear generic}.

If we choose $P =\{0,e_1,\dots,e_m\}$, where $e_1,\dots,e_m$ is the standard basis of $\R^m$, we get:
$$ V_{m,1}(P) = \li(\begin{array}{ccccc}
                       1   & 0   & 0 & \cdots      & 0  \\
                       1  & 1 &0 & \cdots & \vdots\\
                       1 &  0& 1& \dots & 0\\
                       \vdots & \vdots & \ddots& \ddots &\vdots \\
                       1 & 0 & \cdots & 0& 1\\
                      \end{array}\re)\,, \quad 
V_{m,1}^{-1}(P) = \li(\begin{array}{rcccc}
                       1   & 0   & 0 & \cdots      & 0  \\
                       -1  & 1 &0 & \cdots & \vdots\\
                       -1 &  0& 1& \dots & 0\\
                       \vdots & \vdots & \ddots& \ddots &\vdots \\
                       -1 & 0 & \cdots & 0& 1\\
                      \end{array}\re)
                      \,.$$
Thus, $c=(c_0,\dots,c_m)^T \in \R^{m+1}$ with 
$$  c_0 = f(p_1)\,, \quad c_i = f(p_i)-f(p_1)\quad\text{for}\,\,i >1  $$ 
solves $ V_{m,1}(P) c =  F$ with $ F=(f(p_1),\dots,f(p_{m+1}))^T$. Hence, 
\begin{equation}\label{LQ}
  Q(x) =  c_0 + c_1 x_1 + \cdots c_{m} x_m 
\end{equation}
is the unique solution of the PIP with respect to $P$ and the given function $f : \R^m \lo \R$. 
Before solving the linear PIP on the general hyperplanes $H_{\Ic,\xi,b}$ from Definition \ref{Hxi}, we introduce some additional notions.

\begin{proposition}\label{Lin} We are given $m,n \in \mathbb{N}$, a function $f : \R^m \lo \R$, $\xi_1,\dots,\xi_m \in \R^m$ an orthonormal frame, $b \in \R^m$, $\Ic \subseteq \{1,\dots,m\}$, and $H_\Ic:=H_{\Ic,\xi,b}$. 
\begin{enumerate}
 \item[i)]  If $\xi_1,\dots,\xi_m$  is the standard basis, then there exists an algorithm that solves the PIP with respect to $(m,1)$ and $f$ on $H_\Ic$ in $\Oc(k)$, $k=|\Ic|$.  
  \item[ii)] If $\xi_1,\dots,\xi_m$ is an arbitrary frame, then there exists an algorithm that solves the PIP with respect to $(m,1)$ and $f$ on $H_\Ic$ in $\Oc(mk)$, $k=|\Ic|$.
\end{enumerate}

\end{proposition}

\begin{proof} For both cases  we choose an affine transformation $\tau : \R^m \lo \R^m$, $\tau(x)=Ax +b$ 
such that $\tau(H_{1,\dots,k}) =H_\Ic$, $k=|\Ic|$. 
Thus, $H_{1,\dots,k} = i_{1,\dots,k}(\R^k)$ is given by the natural embedding of $\R^k$ into $\R^m$. 
Now we solve the linear PIP on $\R^k$ with respect to $\hat f : \R^k \lo \R$, $\hat f(\hat x) = (f\circ \tau\circ i_{1,\dots,k})(\hat x)$ according to our solution \eqref{LQ} and denote  
$$\hat Q(\hat x)= \hat c_0 + \hat c_1 \hat x_1 + \cdots + \hat c_k\hat x_k\,, \quad \hat c=(\hat c_0,\dots,\hat c_k)$$
the thus determined unique polynomial with generic nodes $\hat P=\{0,\hat e_1,\dots,\hat e_k\}$, $\hat e_i \in \R^k$, $\forall i =1,\dots,k$. 
Now we set $i_{1,\dots,k}(\hat c)=(\hat c_1,\hat c_2,\dots,\hat c_k,0,\dots,0)^T$, $P= \tau(\hat P)$ and 
\begin{align*}
 Q(x) &= \hat c_0 + \li<i_{1,\dots,k}(\hat c),\tau^{-1}(x)\re>= \hat c_0 + \li<i_{1,\dots,k}(\hat c),A^T(x-b)\re> \\
 &= \hat c_0 - \li<A\cdot i_{1,\dots,k}(\hat c),b\re> + \li<A \cdot i_{1,\dots,k}(\hat c),x\re>\,.
\end{align*}
Since $Q(p) = Q(\tau(\hat p)) = \hat Q(\hat p)$ and $\hat f(\hat p)= f(p)$ we find that $Q$ solves the PIP with respect to $(m,1)$ and $P$. Due to Lemma \ref{TT}, $P$ is a generic set of nodes. 
Thus, $Q$ is uniquely determined. Therefore, the coefficients $c=(c_0,\dots,c_m)^T$ of $Q$ are given by 
\begin{equation}\label{lincoff}
 c_0 = \hat c_0 - \li<A\cdot i_{1,\dots,k}(\hat c),b\re> \,, \quad  (c_1,\dots, c_m)^T = A \cdot i_{1,\dots,k}(\hat c)\,.
\end{equation}
  
Due to \eqref{LQ} we can compute $\hat c$ in $\Oc(k)$. If $\xi_1,\dots,\xi_m$ is the standard basis then $A$ is given by permuting the columns of the identity matrix with some permutation 
$\pi \in S_m$. Thus, $\hat c_{\pi^{-1}}(i)$, $i=1,\dots,k$ yields the non-vanishing coefficients of $Q$, which requires $\Oc(k)$ computation steps, hence implying $(i)$.
In the general case, \eqref{lincoff} requires $\Oc(mk)$ computation steps, proving $(ii)$. 
\end{proof}

\section{Algebraic Curves}\label{ALGCurv}
In dimensions \mbox{$>1$}, the non-linear case is much more difficult to solve than either the linear or the 1-dimensional cases. 
We illustrate this by considering the non-linear case in dimension $2$, for which the notation is still tractable. 
For this, however, a deeper understanding of algebraic curves is necessary, which we provide in this section. 

We introduce the concepts of algebraic curves assuming that the reader is familiar with basic algebra.
We also simplify some definitions to the specific notion of our problem setup and adapt them to our notation. For the general definitions, we refer the reader to the excellent reference \cite{fulton}. 
However, an understanding of Corollary \ref{CoBezout} is sufficient to proceed reading this article.

For a field $K$ we denote by $K^m=K\times \dots \times K$ the m-$th$ power of $K$ with itself and by
$K[x_1,\dots,x_m]$ the  ring of polynomials with $m$ variables and coefficients in $K$. We call a polynomial $Q \in K[x_1,\dots,x_m]$ \emph{constant} if and only if
$Q \in K$. 
\begin{definition} Let $K$ be a field, $m \in \N$, and $Q \in K[x_1,\dots,x_m]$ be a  polynomial in $m$ variables.
Then $Q$ is called \emph{irreducible} if and only if for all polynomials $Q_1,Q_2 \in K[x_1,\dots,x_m]$ with $Q =Q_1\cdot Q_2$ at least one of the polynomials $Q_1,Q_2$ is a unit
in $K[x_1,\dots,x_m]$, i.e., 
$Q$ is non-constant and there are no non-constant polynomials $Q_1,Q_2\in K[x_1,\dots,x_m]$ with $Q=Q_1\cdot Q_2$. 
\end{definition}

\begin{definition} Let $K$ be a field, $m \in \N$, and $Q \in K[x_1,\dots,x_m]$ be a non-constant polynomial in $m$ variables. Then we call 
the set 
$$ V(Q)=\li\{ x=(x_1,\dots,x_m) \in K^m \mi Q(x) =0\re\} $$ 
an \emph{algebraic hypersurface}. In the special case $m=2$ we call $V$ a \emph{plane curve}. Given polynomials $Q_1,\dots,Q_k \in K[x_1,\dots,x_m]$, $k \geq 1$, we call $V=V(Q_1,\dots,Q_k)$ given by
$$ V=\li\{ x \in K^m \mi Q_1(x) = Q_2(x)= \cdots = Q_k(x)=0\re\} $$ an \emph{algebraic set}. 
 \end{definition}
Given algebraic sets $V,W\subseteq K^m$ we denote by $I(V)\subseteq K[x_1,\dots,x_m]$ the ideal (in the sense of algebraic ring theory) of all polynomials $Q$ with $Q(x)=0$ for all $x \in V$ and with 
$I(V,W)$ the ideal generated by all polynomials vanishing on $V \cap W$. 
If there are non-empty algebraic sets $V_1 \not = V_2 \subseteq K^m$ with $V=V_1\cup V_2$, then $V$ is called \emph{reducible}. It is called \emph{irreducible} otherwise. 
Note that $V$ is irreducible if and only if $I(V)$ is a prime ideal. Consequently, any hypersurface $V(Q)$ is irreducible if and only if $Q$ is an irreducible polynomial.
Moreover, there holds:

\begin{theorem} Let $V\subseteq K^m$, $m\geq 1$ be an algebraic set. Then, there are unique irreducible algebraic subsets $V_1,\dots,V_k \subseteq V$ such that 
$$ V= V_1\cup \dots \cup V_k\,, \quad V_i \not \subseteq V_j\,, \,\forall\, i \not =j\,.$$ 
 \end{theorem}
The $V_1,\dots,V_k$ are called the \emph{(irreducible) components} of $V$. Thus, the decomposition of an algebraic hypersurface $V(Q)$ into its components is in one-to-one correspondence with the decomposition of $Q$ into its irreducible factors. Furthermore we obtain:
\begin{corollary}\label{planeC} If $K$ is infinite, then the irreducible algebraic sets $V \subseteq K^2$ are $\emptyset$, $K^2$, nodes, and irreducible plane  curves.  
\end{corollary}
Given two plane curves $V,W \subseteq K^2$ we can then ask for the number of their intersections. Due to Corollary \ref{planeC}, we have that 
$\#(V\cap W)=\infty$ whenever $K$ is infinite and $V$, $W$ share a common component. If $V$ and $W$ share no common component, then, as expected, $\#(V\cap W)<\infty$. However, to be more precise we need a concept of counting intersections between $V$ and $W$ 
with multiplicity, e.g., if $V$ intersects both with itself and with $W$ in $p \in K^2$.  
To do so for  $V\subseteq K^2$ and  $p \in V$, we define the 
\emph{local ring} at $p$ by  
$$\mathcal{R}_p=\li\{\frac{Q_1}{Q_2}\mi Q_1,Q_2 \in K[x,y]\,, Q_2(p)\not =0\re\}\,,$$ the set of all rational functions that are well-defined on $p$. Note that $\mathcal{R}_p$ possesses a vector-space structure. Thus, for two plane curves $V,W$ the number 
$$i(p,V\cap W):=\dim\big(\mathcal{R}_p/I(V,W) \big) \in \N$$ 
is a well-defined non-negative integer, called the \emph{intersection number} of $V,W$ at $p$. Indeed, this definition yields an admissible notion of counting intersections

In addition to  
intersections of plane curves, we can also consider projective plane curves, i.e., algebraic curves in the projective plane. This can be done by 
considering $K^3\setminus\{(0,0,0)\}$ and observing that every point $p' \in K^3\setminus\{(0,0,0)\}$ uniquely determines a line through $(0,0,0)$. 
Setting two nodes $p',p'' \in K^3\setminus\{(0,0,0)\}$ to be equivalent if and only if $p',p''$ determine the same line yields an equivalence relation $\sim$. The quotient space 
$\PP^2:=(K^3\setminus\{(0,0,0)\})/_{\sim}  $ is called the \emph{projective plane}. The nodes $p=(x,y,z) \in K^3\setminus\{(0,0,0)\}$ with $z=0$ are interpreted as ``nodes at infinity''. 
Since $K^2 \hookrightarrow \PP^2$ can be embedded into $\PP^2$ by 
mapping $p=(x,y)$ to the equivalence class $p^*$ of $(x,y,1)$, every plane curve $V$ induces a projective plane curve $V^*$. In fact,
this construction yields the observation that every pair of lines intersects in exactly one point, with parallel lines intersecting at ``infinity''. 
While a deeper discussion of these concepts is beyond the scope of this article, we assert that the notion of intersection numbers can be 
adapted to projective curves such that 
$i(p,V\cap W)=i(p^*,V^*\cap W^*)$ holds. This concept has also been used to give a modern formal proof of the following famous result already stated in the $19$-th century \cite{fulton}: 
 \begin{theorem}{\bf (B\'{e}zout's Theorem)} Let $V$ and $W$ be projective plane curves. Assume that $V$ and $W$ have no 
common components. Then
$$ \sum_{p\in \PP^2} i(p,V\cap W) = \deg(V) \cdot \deg(W) \,.$$ 
\label{Bezout}
\end{theorem}

This theorem has an immediate consequence: 
\begin{corollary} Let $V,W \subseteq K^2$ be two  algebraic curves  without common components.
Then 
 $$ \#(V\cap W) \leq  \deg(V)\cdot\deg(W)\,,$$
 where equality holds if and only if $V$ and $W$ do not intersect at infinity and the intersection nodes $p \in (V\cap W) $ are  all \emph{simple nodes}, i.e., nodes with $i(p,V\cap W)=1$.
 \label{CoBezout}
\end{corollary}
\begin{proof} Denoting by $V^*,W^*,p^*$ the corresponding projective curves and nodes, the identity $i(p,V\cap W)=i(p^*,V^*\cap W^*)$ implies 
 \begin{equation}\label{leq}
 \sum_{p\in K^2} i(p,V\cap W) \leq \sum_{Q^* \in \PP^2} i(Q^*,V^*\cap W^*) \,. 
\end{equation}
Since $i(p,V\cap W)  \geq 1$ whenever $p \in (V\cap W)$ the corollary follows from  Theorem \ref{Bezout}.
\end{proof}

In the next section, we show how this corollary can be used to solve Problem \ref{VV} in the non-linear case  in dimension $m=2$.

\section{Vandermonde Systems in Dimension 2}
\label{2D}

We apply Corollary \ref{Bezout} to guarantee that a set of nodes is generic by construction. To do so, the following
 is useful:

\begin{lemma}\label{IR}
For $a,b,c \in \mathbb{R} \setminus \left\{0\right\}$ the polynomial $a x^n + b y^n + c$ is irreducible in $\mathbb{R}\left[x,y\right]$. 
\end{lemma}
\begin{proof}
The constant $a \neq 0$ is a unit in $\mathbb{R} \left[x, y \right]$. Thus, by multiplying $Q$ with  $a^{-1}$   we can assume w.l.o.g.~that $a =1$. 
We then apply the general Eisenstein criterion \cite{Alg}. Therefore,   
we consider $Q := x^n + b y^n + c$ as a polynomial in $x$ with coefficients in $\mathbb{R} \left[y\right] $, namely $ Q  = \sum_{i=0}^n a_i y^i $, where $a_0 = b y^n + c$ and 
$a_1=a_2= \mathellipsis = a_{n-1} = 0$, $a_n = 1 $. 
Since $\mathbb{R} \left[ y \right]$ is an integral domain, we can choose a prime element  $p \in \mathbb{R} \left[ y \right]$ that divides $a_0 = b y^n + c$. 
However, $p^2$ does not divide $b y^n + c  = b \left( y^n + c/b \right) $ for the following reason: 
The roots of $y^n + c/b$ seen as a polynomial in $\mathbb{C}[y]$ are the $n$-th roots of unity (scaled with the common factor $\lvert c/b \rvert ^{1/n}$). Since 
every root appears with  multiplicity $1$, we observe that $b \left(y^n + c/b \right)$ has pairwise-different irreducible factors in $\mathbb{C} \left[y\right]$, which implies that 
$b \left(y^n + c/b \right)$ has pairwise-different irreducible or prime factors in $\mathbb{R}\left[y\right]$.
Summarizing,  the prime element $p$ fulfills 
\begin{align*}
p \mid a_i \; (i<n),  \qquad   p \nmid a_n  ,  \qquad  p^2 \nmid a_0 .
\end{align*}  
In addition, $Q$ is primitive in $\left( \mathbb{R} \left[y\right] \right) \left[ X \right]$, i.e., a common divisor $d \in \mathbb{R} \left[y\right]$ of \linebreak
$a_0, a_1, \mathellipsis, a_n$ has to divide $a_n = 1$, implying that $d$ must be a unit.
Hence, by Eisenstein's criterion, the polynomial $Q$ is irreducible in $\R[x,y]$. 
\end{proof}

Since  $\mathbb{R} [y]$ is a unique factorization domain, we note that $Q$ is even irreducible in $\mathbb{R} ( y ) [ x ]$,
where $\mathbb{R} ( y )$ denotes the field of fractions of the ring $\mathbb{R} [ y ]$. 
Before we construct a general class of generic nodes, we want to consider the simplest non-linear case. 

\begin{ex} Let $L_1,L_2 \subseteq \R^2$ be two transversal lines, i.e., $L_1,L_2$ intersect in exactly one point. Assume there are three nodes $p_1,p_2,p_3 \in L_1$,   
two nodes $p_4,p_5 \in L_2\setminus L_1$, and one point $p_6 \not \in (L_1\cup L_2)$. Certainly, lines are irreducible plane curves. Thus, due to Corollary \ref{CoBezout}, any other  plane curve 
$V \subseteq \R^2$, $V\not = L_1$ of degree \mbox{$\leq 2$} intersects $L_1$ in at most $2$ nodes. Hence, $L_1$ is the only plane curve of degree \mbox{$\leq 2$} containing $p_1,p_2,p_3$.
Since $p_4,p_5,p_6$ can not be contained in one line, the set  $P=\{p_1,\dots,p_6\}$ can not be the zero set of a polynomial of degree \mbox{$\leq 2$} and is therefore generic due to Theorem \ref{generic}.
\label{line} 
\end{ex}

\begin{ex}\label{Ellipse} Assume there are five nodes $P'=\{p_1,\dots,p_5\} \subseteq \R^2$ sitting on an ellipse $E$ given as the zero set of
$Q(x,y)= \frac{x^2}{a^2}+\frac{y^2}{b^2}-1$, $a,b \in \R\setminus\{0\}$, i.e., 
$Q(p_i) = 0$, $\forall \, i=1,\dots,5$. By Lemma \ref{IR} we have that $Q$ is irreducible, which implies, due to Corollary \ref{CoBezout}, that  any other  plane curve 
$V \subseteq \R^2$, $V\not = E$ of degree \mbox{$\leq 2$} intersects $E$ in at most 4 nodes. Hence, $E$ is the only curve of degree \mbox{$\leq 2$} that contains $P'$. 
Therefore, by choosing any $p_6$ outside the ellipse, e.g.~setting $p_6=(0,0)$, Theorem \ref{generic} guarantees that the set $P=\{p_1,\dots,p_6\}$ is generic. 
 \end{ex}

The Examples \ref{line},\ref{Ellipse} suggest the following generalization: 
 
 \begin{theorem} Let $n \in \N$ and $V_1,\dots,V_k$, $k\in \N$ be irreducible algebraic plane curves of degrees $\deg(V_i)=d_i \leq n$, $d_i \geq d_j$, $1 \leq i \leq j \leq k$, such that $\sum_{i=1}^kd_i >n$. 
  Furthermore let $P_i \subseteq V_i$ be finite sets of nodes with 
  \begin{enumerate}
   \item[i)] $P_i \cap P_j = \emptyset \quad \forall i,j: 1\leq i < j \leq k$ 
   \item[ii)] $\#P_i> d_i\cdot \big(n-\sum^{i-1}_{j=1}d_j\big)$
   \item[iii)] $\sum_{i=1}^k \#P_i = N(2,n)$\,.
  \end{enumerate}
Then $P:= \cup_{i=1}^kP_i$ is a generic set of nodes, i.e., the Vandermonde matrix $V_{2,n}(P)$ is regular.
\label{G2}
   \end{theorem}

\begin{proof}  
Since all $V_i$ are irreducible, by setting $W_i =\cup_{j=1}^i V_j$ and $Q_i= \cup_{j=1}^iP_j$   
Corollary \ref{CoBezout} implies that  there is no curve $V \not = W_i$ of degree $\deg(V)\leq  \sum_{j=1}^{i-1} d_j$ with 
$ V \cap W_i  \supseteq Q_i$, implying that there is no curve $V'$ of degree $\deg(V') \leq n$ containing $P$. Thus, $P$ is generic due to Theorem \ref{generic}.
\end{proof}

Certainly, we can choose the irreducible algebraic sets $V_i$ to be  the zero sets of the polynomials considered in Lemma \ref{IR}. In this case 
we can assume that  one of the parameters $a,b,c$ equals $1$ and therefore the $V_i$ 
depend only on the degree and two additional parameters, which simplifies their computation. Thus, Theorem \ref{G2} provides a recursive procedure for constructing generic sets in dimension \mbox{$m=2$}. 
If in addition the algebraic curves differ significantly, and the nodes are distributed homogeneously on the curves, i.e., the curves intersect almost orthogonally and the nodes are not too close to each other
and not too far from each other,
then inversion of the Vandermonde matrix is numerically accurate and so will be the derived solution of the PIP.

\section{PIP-SOLVER of the General PIP}\label{Decomp}
In dimensions $m \geq 2$ we establish a decomposition of the problem into two sub-problems, one of lower dimension \mbox{$(m-1,n)$} and 
one of lower degree \mbox{$(m,n-1)$}. 


\begin{theorem} Let $m,n \in \N$, $m \geq 1$ and $P \subseteq \R^m$ such that:
\begin{enumerate}
 \item[i)] There is a hyperplane $H \subseteq \R^m$ of co-dimension 1 and $P_1:= P \cap H$ satisfies $\#P_1 = N(m-1,n)$ and is generic with respect to $H$, i.e., 
 by identifying $H\cong \R^{m-1}$ the Vandermonde matrix $V_{m-1,n}(P_1)$ is regular.
  \item[ii)]  The set $P_2= P \setminus H$ satisfies $\#P_2 = N(m,n-1)$ and is generic with respect to the parameters \mbox{$(m,n-1)$}, i.e., the Vandermonde matrix $V_{m,n-1}(P_2)$ is regular.
\end{enumerate}
Then $P$ is a generic set. 
\label{GN}
\end{theorem}

\begin{proof} 
Due to Lemma \ref{TT}, generic node sets remain generic under affine transformation. Thus, by choosing the appropriate transformation $\tau$,  we can assume  w.l.o.g.~that $H= H_{1,\dots, m-1}$. 
For any polynomial $Q \in \R[x_1,\dots,x_m]$ there holds 
$Q\big(\pi_{1,\dots,m-1}(p)\big) = \pi_{1,\dots,m-1}^*\big(Q(p)\big)$ for all $p \in H$. Thus,  by $(i)$ we observe that  
whenever there is a $Q \in \R[x_1,\dots,x_m]$ with $Q(P) =0$, then $\deg(Q) \geq \deg\big(\pi_{1,\dots,m-1}^*(Q)\big) > n$. Or, if $\deg\big(\pi_{1,\dots,m-1}^*(Q)\big) \leq \deg(Q) \leq n$, we consider 
$\bar Q_1:= Q - i^*_{1,\dots,m-1}\big(\pi_{1,\dots,m-1}^*(Q)\big)$, which consists of all monomials sharing the variable $x_m$ and $\bar Q_2:= Q -\bar Q_1$ consisting of all monomials not sharing $x_m$.  
We claim that $\bar Q_2 = 0$. Certainly, $\bar Q_1(x) = 0$ for all $x \in H$.  Since $P_1$ is generic there are $p \in P_1$ with $\bar Q_2(p)\not = 0$ implying $Q(p) \not = 0$,
which contradicts our assumption on $Q$ and therefore yields $\bar Q_2 = 0$ as claimed. 
In light of this fact, and due to the genericity of $P_1$ w.r.t. $H$ we get that
$Q$ can be decomposed into polynomials $Q=Q_1\cdot Q_2 \quad \text{where} \quad   Q_2(x_1,\dots,x_m) =x_m\,.$ Since $\deg(Q_1)\leq n-1$ we have that $P_2 \not \subseteq Q_1^{-1}(0)$ due the genericity of $P_2$.
At the same time, $P_2 \cap H =\emptyset $ implies that $Q_2(p) \not = 0$ for all $p \in P_2$. Hence, there is $p \in P_2$ with $Q(p)\not =0$, proving the claim due to Theorem \ref{generic}. 
\end{proof}

\begin{remark}
For \mbox{$m=n=2$}, Theorem \ref{GN} reflects the situation of Example \ref{line}.  
  \end{remark}

\begin{ex} Let \mbox{$m=3$}, \mbox{$n=2$} and assume that we have determined a plane $H \subseteq \R^3$ and  a generic set $P_1\subseteq H$  with $\# P_1 =6$ by
following Example \ref{Ellipse}. Then we choose $P_2 \subseteq \R^3\setminus H$ such that $P_2$ is linear generic and observe that  $P=P_1 \cup P_2$ is  a  generic set due to Theorem \ref{GN}.
 \label{3D}
\end{ex}

The question arises whether the decomposition of algebraic curves given in Theorem \ref{GN} allows us to decompose the PIP into smaller and therefore simpler, solvable sub-problems.
This is indeed the case: 

\begin{theorem} \label{SV} Let $m,n \in \N$, $m \geq 1$,  $f : \R^m\lo \R$ a given function, $H \subseteq \R^m$ a hyperplane of co-dimension 1, $Q_H \in \R[x_1,\dots,x_m]$ a polynomial such that $\deg(Q_H)=1$ and $Q_H^{-1}(0) =H$, 
and $P=P_1\cup P_2\subseteq \R^m$ such that $(i)$ and $(ii)$ of Theorem \ref{GN} hold with respect to $H$.
Require $Q_1,Q_2 \in \R[x_1,\dots,x_m]$ to be such that: 
\begin{enumerate}
 \item [i)] $Q_1$ has degree $\deg(Q_1)\leq n$ and solves the PIP with respect to $f$ and $P_1$ on $H$.
 \item [ii)] $Q_2$ has degree $\deg(Q_2)\leq n-1$ and solves the PIP with respect to $\hat f:=(f-Q_1)/Q_H$ and $P_2=P\setminus H$ on $\R^m$.
\end{enumerate}
Then $Q=Q_1+Q_HQ_2$ is the uniquely determined polynomial with $\deg(Q)\leq n$ that solves the PIP with respect to $f$ and $P$ on $\R^m$.
\end{theorem}
\begin{proof}  By our assumption on $Q_H$ and $Q_1$ we have that $Q(x)=Q_1(x)$ $\forall x \in H$ and therefore $Q(p) =f(p)$ $\forall p \in P_1 = P \cap H$. 
At the same time, $Q_H(x) \not =0$ for all $x \not \in H$. Therefore $\hat f$ is well defined, and $Q(p)= Q_1(p) + \hat f(p)Q_H(p)=f(p)$ $\forall p \in P_2=P\setminus H$. Hence $\deg(Q)\leq n$, and $Q$ solves the PIP with respect to $f$ 
and a generic set of nodes $P\subseteq \R^m$. Thus, $Q$ is the unique solution of the PIP with respect to $P$ and $f$. 
 \end{proof}
 
 \begin{remark}\label{normal}
 Note that $Q_H$ can be constructed by choosing a (usually unit) normal vector $\nu \in \R^m$ of $H$ and a vector $b \in H$, setting 
 $$Q_H(x) = \nu\cdot(x - b)\,.$$ Indeed, we have that $Q_H(x)=0$ for all $x \in H$ and 
 $Q_H(x)\not =0$ for all $x \in R^m\setminus H$.
  \end{remark}

We close this section by stating the main result in a more precise way and delivering its proof. 
 \begin{theo}[Main Result] Let $m,n \in \N$, $m\geq 1$, and $f : \R^m \lo \R$ be a given function. Then there exists an algorithm with runtime complexity $\Oc\big(N(m,n)^2\big)$,  requiring storage in $\Oc\big(mN(m,n)\big)$, that computes: 
 \begin{enumerate}
  \item [i)] a generic node set $P \subseteq \R^m$ with respect to $m,n$;
  \item[ii)] a polynomial $Q \in \R[x_1,\dots,x_m]$ with $\deg (Q) \leq n$, such that $Q$ is the unique solution of the PIP with respect to $f$ and 
  $P'=P+\mu$, where $P$ was generated in $(i)$ and $\mu \in \R^m$ is an arbitrary vector.  
 \end{enumerate}
\label{mainT}
 \end{theo}
 
 \begin{proof} We start by proving $(i)$ and $(ii)$ with respect to the runtime complexity. To do so, we  claim that there is a constant $C\in \R^+$ and an algorithm computing  $(i)$ and $(ii)$ in less than $CN(\mbox{m,n})^2$ computation steps. To prove this claim, we argue by induction on $N(m,n)$.
 If $N(m,n)=1$, then $n=0$. According to Section \ref{SC}, the claim then holds. Now let $N(m,n)>1$. If $m=1$ or $n=1$ Propositions \ref{1D} and \ref{Lin}, provide us with the required runtime bounds. Thus, the claim holds by choosing $C$ as the largest occurring constant.
 If $m>1$ and $n>1$ then we choose $\nu,b\in \R^m$ with $||\nu ||=1$ and consider the hyperplane $H=Q_H^{-1}(0)$, $Q_H(x)=\nu(x-b)$.
 By identifying $H\cong \R^{m-1}$, induction yields that we can determine a set of generic nodes $P_1\subseteq H$ in less than \mbox{$CN(m-1,n)^2$} computation steps. 
 Induction also yields that a set $P_2 \subseteq R^m$ of generic nodes can be determined with respect to $n-1$ in less than
 \mbox{$CN(m-1,n)^2$} computation steps. By translating $P_2$ with $\lambda \nu$, i.e., setting $P_2'=P_2 + \lambda \nu $, $\lambda \in \R$, we can guarantee that 
 $P_2 \cap H = \emptyset$. Hence, the union $P=P_1 \cup P_2'$ of the corresponding generic sets of nodes is also generic due to Theorem \ref{GN}, proving $(i)$. 
 
 It is $N(m-1,n)+N(m,n-1)=N(m,n)$. Thus, by $(ii)$ and induction, polynomials $Q_1,Q_2 \in \R[x_1,\dots,x_m]$, $\deg(Q_1)=n$, $\deg(Q_2)=n-1$ can be determined in less than 
 \mbox{$CN(m-1,n)^2+CN(m,n-1)^2 \leq CN(m,n)^2$} computation steps, such that $Q_1$ solves the PIP on $H$ with respect to $f$ and $P_1$,
 while $Q_2$ solves the PIP with respect to $\hat f = (f-Q_1)/Q_H$ and $P_2'$. Due to Theorem \ref{SV} we have that $Q_1+Q_HQ_2$ solves the PIP with respect to $f$ and $P$. Thus, for any $\mu \in \R^m$, setting $f'(x)=f(x+\mu)$ yields a solution 
 with respect to $f$ and $P'=P+\mu$.  Therefore, it remains to bound the steps required for computing $Q_1+Q_HQ_2$. The bottleneck herin lies in the computation of $Q_HQ_2$, which requires $C_2(m+1)N(m,n-1)$, $C_2 \in \R^+$ 
 computation steps. Observe that $2N(m-1,2) =m(m+1)$ and $N(m-1,k)\leq N(m-1,k')$ for $k\leq k' \in \N$.
 Thus, $(m+1)N(m,n-1) \leq 2N(m-1,n)N(m,n-1)$, for $n>1$, which shows that $Q_1+Q_HQ_2$ can be computed in less than $C_3N(m-1,n)N(m,n-1)$, $C_3 \in \R^+$ computation steps. 
 Hence, by using again $N(m-1,n)+N(m,n-1)=N(m,n)$ and assuming $2C\geq C_3$, we have that 
 $$ C\big(N(m-1,n)^2 + N(m,n-1)^2\big) + C_3N(m-1,n)N(m,n-1) \leq CN(m,n)^2 \,,$$
 proving $(ii)$. 
 
 We obtain the storage complexity by using an analogous induction argument. If $N(m,n)=1$ we need to store at most $CmN(m,n) =D$, $D \in \N$, numbers. If 
 $N(m,n)>1$ and  $m=1$ or $n=1$, Propositions \ref{1D} and \ref{Lin} imply that we have to store the generic nodes, which requires $D_1mN(m,n)$ numbers, and the coefficients, requiring $D_2N(m,n)$ numbers.
 This shows the claim by setting $D=\max\{D_1,D_2\}$.
 If  $N(m,n)>1$ and  $m,n>1$, using the same splitting of the problem as above, induction yields that we have to store at most $D(m-1)N(m-1,n)$, $DmN(m,n-1)$ numbers for each sub-problem. Thus, altogether we need to store 
$ D(m-1)N(m-1,n) + DmN(m,n-1) \leq DmN(m,n)$ numbers, proving the storage complexity.

%
%
%
 \end{proof}

Summarizing our results so far, we have provided  a sufficiently deep understanding of the problem to be able to implement an algorithm that efficiently solves the PIP in arbitrary dimension and order.
A more detailed description of this algorithm is given in the next section.

\section{Algorithm}\label{Experiments}  
As illustrated in Example \ref{3D}, Theorem \ref{mainT} enables us to construct a generic set by determining a hyperplane $H_0$,  a generic set 
 $P_1 \subseteq H_0$
with respect to the parameters \mbox{$(m-1,n)$}, and a generic set $P_2$ with respect to the parameters \mbox{$(m,n-1)$} on $H=\R^m$, such that $H_0 \cap P_2 = \emptyset$. Iterating
this decomposition 
yields a tree $T_{H,Q}$ of sub-problems with the root corresponding to the global problem with respect to $(m,n)$ and leaves corresponding to linear or $1$-dimensional sub-problems. 
Thus, by providing that the requirements of Theorem \ref{SV} are fulfilled for recursive decomposition step, we are able to solve the PIP by composing  solutions of only $1$-dimensional and linear sub-problems. We exploit this to design an algorithm, termed PIP-SOLVER, that solves the general PIP efficiently and numerically accurately.  

\begin{algorithm}[t!]
\caption{Build $T_{H,Q}$}
{\bf Input:} $(T_{m,n},\ee,\sigma)$, frame $F= \{\xi_1,\dots,\xi_m\}$; 
\begin{algorithmic}
\STATE{$\xi_{m+1} \leftarrow 0$, $b \leftarrow 0 $;}
\FOR{ $k= 1$ {\bf to} $D(T_{m,n})-1$}
\WHILE{$\Vc_{m,n}^1(k) \not = \emptyset$}
\STATE{Choose $v \in \Vc_{m,n}^1(k)$;}
\STATE{ $\nu_\ee \leftarrow   \xi_{\sigma_1(v)+1}$, $b_\ee \leftarrow  b_{\hat \ee} + \alpha(\ee(v))\nu_{\ee}$;}
\STATE{$\Vc_{m,n}^1(k) \leftarrow \Vc_{m,n}^1(k) \setminus \{v\}$;}
\ENDWHILE
\STATE{$k \leftarrow k+1$}
\ENDFOR
\RETURN List of Polynomials $Q_{H_\ee}(x)=\nu_\ee(x -b_\ee)$;
\end{algorithmic}
\label{A1}
\end{algorithm}

\begin{definition}\label{TMN} Let $m,n \in \mathbb{N}$, $m,n\geq 2$. We define a tree $T_{m,n}=(\Vc_{m,n},\Ec_{m,n})$ as follows: We start with a root $r$ and  let $D(v) \in \N$ denote the depth of vertex 
$v \in \Vc_{m,n}$, i.e., 
the shortest-path distance from $v$ to the root $r$. Let $\Vc_{m,n}(k)=\li\{v \in \Vc_{m,n} \mi K(v)=k\re\}$ be the set of all vertices at depth $k$. Thus, $\Vc_{m,n}(0)=\{r\}$. Additionally,
we denote by $S\subseteq \{0,1\}^{\N}$ the set of all  finite sequences with values in $\{0,1\}$ and define 
$$ \ee : \Vc_{m,n} \lo S\,, \quad \sigma : \Vc_{m,n}(k)  \lo \N\times \N $$
recursively by setting $\sigma(r)=(m,n)$, $\ee(r) = \emptyset$ and introducing two children $v,u \in \Vc_{m,n}(k)$ if the parent $w \in \Vc_{m,n}(k-1)$ satisfies $\sigma_1(w) >1$ and $\sigma_2(w) >1$. 
We set $\ee(v)=(\ee(w),0)$,  $\ee(u) = (\ee(w),1)$ and label them with $\sigma(v) = (\sigma_1(v),\sigma_{2}(v))=(m,n-1)$, $\sigma(u) = (\sigma_1(u),\sigma_{2}(u))=(m-1,n)$, respectively.
We denote the resulting binary, labeled, enumerated, rooted tree   by $(T_{m,n},\ee,\sigma)$.
\end{definition}

Example \ref{AS} illustrates the construction of such a tree and of how Algorithms \ref{A1}, \ref{P}, and \ref{A2} use the data structure $T_{m,n}$ to recursively solve the PIP.
For a better understanding of this, the following facts and notions are useful: 

The depth of a vertex $v \in \Vc_{m,n}$  is given by $D(v) = m-\sigma_1(v) + n - \sigma_2(v)$, and the total depth of the tree $T_{m,n}$ is denoted by $D(T_{m,n})$.
We let $L_{m,n}$ be the set of all leaves
and observe:
\begin{lemma} Let $m,n \in \N$ and $T_{m,n}$ be given. 
\begin{enumerate}
 \item[i)] The tree depth is given by $D(T_{m,n})=m+n-2$. 
 \item[ii)] The number $\# L_{m,n}$  of all leaves of $T_{m,n}$ is given by $N(m-1,n-1)$. 
\end{enumerate}
\end{lemma}
\begin{proof} Point $(i)$ follows directly by Definition \ref{TMN}. To see $(ii)$, we argue by induction on $m+n$. Indeed $\#L_{2,2}=2=N(1,1)$. If $m+n>4$ then we consider the subtrees $T_{m-1,n}$, $T_{m,n-1}$ rooted at the vertices 
$u,v \in V_{m,n}$ with labels $\sigma(u)=(m,n-1)$, $\sigma(v)=(m-1,n)$. By induction we have $\#L_{m-1,n}=N(m-2,n-1)$ and $\#L_{m,n-1}=N(m-1,n-2)$. Now we use 
$N(m-1,n)+N(m,n-1)= N(m,n)$ to verify that $\#L_{m,n}= \#L_{m-1}+\#L_{m,n-1} = N(m-2,n-1)+ N(m-1,n-2)= N(m-1,n-1)$.  
 \end{proof}
Furthermore, we consider:  
$$\Vc_{m,n}^1(k) = \li\{v \in \Vc_{m,n}(k) \mi \text{ the last entry of $\ee(v)$ equals $1$} \re\}\,.$$ 
By  $\hat \ee(v)$ we denote the sequence obtained by deleting the last entry of $\ee(v)$. Moreover, we denote by $F=\{\xi_1,\dots,\xi_m\} \subseteq \R^m$ an orthonormal frame of $\R^m$, i.e., $\xi_i \perp \xi_j$ $\forall i\not =j$, $1 \leq i,j \leq m$ and introduce $\alpha : S \lo \Z$ as the enumeration from dual system to rationals by: 
$$ \alpha(\ee) =  \sum_{i=1}^{|\ee|}(-1)^{i-1}\ee_i\lambda^{i}\,,\quad \lambda \in \mathbb{Q}\,.$$
Algorithm \ref{A1} uses the function $\alpha$  to generate hyperplanes $H_{\ee}$ by shifting planes $\widetilde H_\ee \mapsto H_\ee$, $\widetilde H_\ee  \perp \xi_{k+1},\dots,\xi_m$  such that 
$H_{\hat \ee} \subseteq H_{\ee}$ for all $\ee$. Due to the unique enumeration $\ee(v)$ of the vertices $v \in \Vc_{m,n}$ and the construction of $\alpha$ we observe that for suitable choice of $\lambda$:
\begin{equation}\label{intersect}
F \cap H_{\ee} = \emptyset , \,\,  \forall \ee \not = \emptyset \quad \text{and} \quad H_{\ee} \cap H_{\ee'} =\emptyset \text{ whenever }\dim H_{\ee} = \dim H_{\ee'}\,.
 \end{equation}
Consequently, Algorithm \ref{P} determines generic sets independently for  each sub-problem occurring on the leaves of $T_{m,n}$. Thus, we can  choose \emph{Chebyshev nodes} on the 1-dimensional plane $H_\ee$ given by 
$H_\ee=\li\{\mu\xi_1 + b_\ee \mi \mu \in \R\re\}$, i.e., we set 
$$ \mathrm{cheb}_{F,\kappa}(n,b)=\kappa\cos\li(\frac{2k-1}{2n}\pi\re) \xi_1 + b\,,  \quad k=1,\dots,n\,, b \in \R^m\,, \kappa \in \R^+\,. $$
Generic sets for the liner sub-problems are determined as described in Section \ref{linear}. 
\begin{algorithm}[t!]
\caption{Determine generic sets $P_v$, $v \in L(T_{m,n})$. }
{\bf Input:} $f$, $(T_{m,n},\ee,\sigma)$, frame $F= \{\xi_1,\dots,\xi_m\}$, $b_\ee$;
\begin{algorithmic}
\STATE{$b \leftarrow 0$;}
\WHILE{$L(T_{m,n}) \not = \emptyset$}
\STATE{Choose $v \in L(T_{m,n})$;}
\IF{$\sigma_1(v)=1$}
\STATE{ $P_v \leftarrow \mathrm{cheb}_F\big(\sigma_2(v),b_{\ee(v)}\big) $;}
\ELSIF{$\sigma_2(v)=1$}
\STATE{ $p_1 \leftarrow b_{\ee(v)}$, $p_i \leftarrow \xi_{i} + b_{\ee(v)} $, $i =2,\dots,\sigma_1(v)$, $P_v = \{p_1,\dots,p_{\sigma_1(v)}\}$;}
\ENDIF
\STATE{$L(T_{m,n}) \leftarrow L(T_{m,n}) \setminus \{v\}$;}
\ENDWHILE
\RETURN List of generic sets $P_v$.
\end{algorithmic}
\label{P}
\end{algorithm}
 \begin{algorithm}[t!]
\caption{PIP-SOLVER: compute the solution $Q$ of the PIP w.r.t.~$m,n$ and $f$ on $H=\R^m$.}
{\bf Input:} $f, (T_{m,n},\ee,\sigma), Q_{H_\ee}$ 
\begin{algorithmic}
\STATE{$\hat f\leftarrow f$}
\STATE{$v = v \in L_{m,n}$ leaf with $\sigma(v) =(1,n)$}
\STATE{$\hat Q_v \leftarrow 0$}
\WHILE{finish $=0$}
\IF{$\sigma_1(v)=1$}
\STATE{Compute the generic nodes $P_v$ and the solution $Q_v$ of the 1-dimensional PIP w.r.t.~$P_v$ and $\hat f = (f -\hat Q_v)/\hat Q_{H,v}$}
\STATE{$\mathrm{problem}(v)\leftarrow \mathrm{solved}$}
\STATE{$w\leftarrow \mathrm{parent}(v)$}
\STATE{$\hat Q_w \leftarrow \hat Q_v + Q_v \cdot \hat Q_{H,v}$}
\STATE{$\hat Q_{H,w} \leftarrow  \hat Q_{H,v}$}
\STATE{$v\leftarrow w$.}
\ELSIF{$\sigma_2(v)=1$}
\STATE{Compute the generic nodes $P_v$ and the solution $Q_v$ of the linear PIP w.r.t.~$P_v$ and $\hat f = (f -\hat Q_v)/\hat Q_{H,v}$}
\STATE{$\mathrm{problem}(v)\leftarrow \mathrm{solved}$}
\STATE{$w\leftarrow \mathrm{parent}(v)$}
\STATE{$\hat Q_w \leftarrow  (\hat Q_v + Q_v \cdot \hat Q_{H,v})\cdot Q_{H_{\ee(v)}}$}
\STATE{$\hat Q_{H,w} \leftarrow  \hat Q_{H,v}\cdot Q_{H_{\ee(v)}}$}
\STATE{$v\leftarrow w$.}
\ELSIF{$\mathrm{problem}(u)\not= \mathrm{solved}$, $u=\text{left-children}(v)$}
\STATE{$\hat Q_u \leftarrow \hat Q_v$}
\STATE{$ \hat Q_{H,u} \leftarrow \hat Q_{H,v}$}
\STATE{$v\leftarrow u$.}
\ELSIF{$\mathrm{problem}(w)\not= \mathrm{solved}$, $w=\text{right-children}(v)$}
\STATE{$\hat Q_u \leftarrow \hat Q_v + Q_v\cdot \hat Q_{H,v}\cdot Q_{H,v}$}
\STATE{$ \hat Q_{H,u} \leftarrow \hat Q_{H,v}\cdot Q_{H,v}$}
\STATE{$v\leftarrow u$.}
\ELSIF{$\mathrm{problem}(u)= \mathrm{problem}(w)= \mathrm{solved}$, $u,w=\text{left/right-children}(v)$}
\STATE{$Q_v \leftarrow Q_u + Q_w\cdot Q_{H,w}$}
\STATE{$\mathrm{problem}(v)\leftarrow \mathrm{solved}$}
\IF{$v=\mathrm{root}$}
\STATE{$\mathrm{finish} \leftarrow 1$}
\ELSE
\STATE{$v\leftarrow \mathrm{parent}(v)$}
\ENDIF
\ENDIF
\ENDWHILE
\RETURN $Q_{v}$;
\end{algorithmic}
\label{A2}
\end{algorithm}
The computations are explained in the following example.  

\begin{ex}\label{AS} Let $m=n=3$ and $F =\{e_1,e_2,e_3\}$. Then, following Algorithm \ref{A1}, we find that the planes 
\begin{align*}
& H_0 = \R^3, \,\,  H_1 =  H_{x,y}+2e_3, \,\, H_{0,0} =   \R^3, \,\, H_{0,1}=  H_{x,y}-4e_3, \,\, H_{1,0} =   H_1, \\ 
& H_{1,1}=  H_x +2e_3 -2e_2,\,\, H_{0,1,1}=  H_x -4e_3 +4e_2, \,\,  H_{1,0,1} =   H_x +2e_3 +8e_2 \,.
\end{align*}
can be determined by $Q_{H_0}(x)=0$, $Q_{H_{\ee,0}}(x)=Q_{H_{\ee}}(x)$, $\ee \in  S \subseteq \{0,1\}^{\N}$ and the polynomials
\begin{align*}
 & Q_{H_1}(x) = e_3(x -2e_3), \,\, Q_{H_{0,1}} = e_3(x + 4e_3), \,\, Q_{H_{1,1}}(x) = e_2(x -2e_3+2e_2), \\ 
 & Q_{H_{0,1,1}}(x) = e_2(x +4e_3-4e_2), \,\, Q_{H_{1,0,1}}(x) = e_2(x-2e_3-8e_2). 
\end{align*}
For instance, $H_{0,1,1} = \li\{x \in \R^3 \mi Q_{H_0}(x)=Q_{H_{0,1}}(x) = Q_{H_{0,1,1}}(x) =0 \re\}$ with $Q_{H_0}(x)=0$. Indeed, no planes of equal dimension intersect.
Therefore, Algorithm \ref{P} can use the procedure of Section \ref{SC} to generate generic disjoint sets $P_v$, $v \in L(T_{3,3})$, for every sub-problem independently.
\end{ex}

  \begin{figure}[t!]
  \centering
  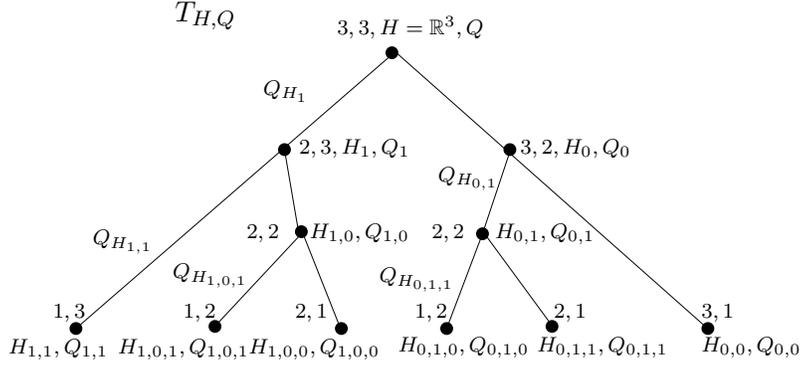
  \caption{The PIP decomposition tree $T_{H,Q}$ for $m=n=3$.}\label{Tree}
\end{figure}

Now that the generic nodes are determined, we use Propositions \ref{1D},\ref{Lin} to compute the interpolating polynomial as follows: 
Consider $T_{H,Q}$ and start at the leaf $v \in L_{m,n}$ with $\sigma(v)=(1,n)$.  Solve the 1-dimensional PIP 
with respect to $f$ and $P_v$. Then, proceed from left to right to the next leaf $u$ and solve the 1-dimensional or linear PIP with respect to the corresponding corrected version $\hat f_u$ of $f$ according to Theorem \ref{SV}.
When done with all leaves, compute the global solution by following Theorem \ref{SV} in combining the leaf sub-solutions. Algorithm \ref{A2} formalizes this  procedure.

\begin{ex} We continue Example \ref{AS} by computing the interpolating polynomial for the generic nodes determined therein. Starting in $v_1$ with $\sigma(v_1)=(1,3)$, we compute a generic set of nodes $P_{v_1}$ by following Proposition \ref{1D} and the solution $Q_{1,1}$
of the $1$-dimensional PIP with respect to $f$ and $P_{v_1}$. Then we go to the next leaf $v_2$ with $\sigma(v_2)=(1,2)$ and compute $P_{v_2}$ and $Q_{1,0,1}$ with respect to $\hat f= (f - Q_{1,1})/Q_{H_{1}}$ and $P_{v_2}$.
We proceed to $v_3$ with $\sigma(v_3)=(2,1)$ compute linear generic nodes $P_{v_3}$ and $Q_{1,0,0}$ with respect to $\hat f = (f - Q_{1,1}-Q_{1,0,1}\cdot Q_{H_{1,1}})/(Q_{H_{1}}\cdot Q_{H_{1,1}})$.
Then, we can compute $Q_1= Q_{1,1} + Q_{H_{1,1}}Q_{1,0}$, $Q_{1,0}= Q_{1,0,1} + Q_{H_{1,0,1}}Q_{1,0,0}$ and proceed to $v_4$ with $\sigma(v_4) =  (1,2)$, solving the $1$-dimensional PIP with respect to 
\begin{align*}
 \hat f = &(f -Q_1)/Q_{H_1} \\
 =& (f - Q_{1,1}-Q_{1,0,1}\cdot Q_{H_{1,1}} - Q_{1,0,0}\cdot Q_{H_{1,1}}\cdot Q_{H_{1,0,1}})/Q_{H_{1}}.
\end{align*}
We solve the PIP for the remaining leaves analogously, observing that finally 
\begin{eqnarray*}
 Q &=& Q_1  + Q_{H_1}Q_0 = Q_{1,1} + Q_{H_{1,1}}Q_{1,0} + Q_{H_1}(Q_{0,1}+Q_{H_{0,1}}Q_{0,0}) \\ 
   &=& Q_{1,1} + Q_{H_{1,1}}(Q_{1,0,1} + Q_{H_{1,0,1}}Q_{1,0,0}) + Q_{H_1}\big((Q_{0,1,1} + Q_{H_{0,1,1}}Q_{0,1,0}) + Q_{H_{0,1}}Q_{0,0}\big) \nonumber
\end{eqnarray*}
solves the PIP with respect to $f$ and $P$ on $\R^3$.   
\end{ex}

 \begin{remark}\label{Q} Note that the results of Section \ref{2D} enable us to stop the iteration of Algorithms \ref{A1}, \ref{P}, \ref{A2} 
whenever we drop into a $2$-dimensional sub-problem. As outlined in Section \ref{2D} and Example \ref{3D},  this case can also be directly solved by generating a generic sets of nodes and inverting the corresponding small Vandermonde matrix 
$V_{2,k} \in \R^{K\times K}$, $K=(k+1)(k+2)$. 
 \end{remark}

 \section{Numerical Experiments} \label{EX}
 We verify and illustrate some of the findings of the previous sections in numerical experiments. 
 For this, we use a prototype MATLAB (version: R2015b (8.6.0.267246)) implementation of the algorithms described above running on an Apple MacBook Pro (Retina, 15-inch, Mid 2015)
 with a 2.2\,GHz Intel Core i7 processor and 16\,GB 1600\,MHz DDR3 memory and operating system macOS Sierra (version 10.12.14.).

\begin{figure}[t!]
\begin{minipage}[t]{0.47\textwidth}
\vspace{-2.05cm}
  \includegraphics[scale=0.4]{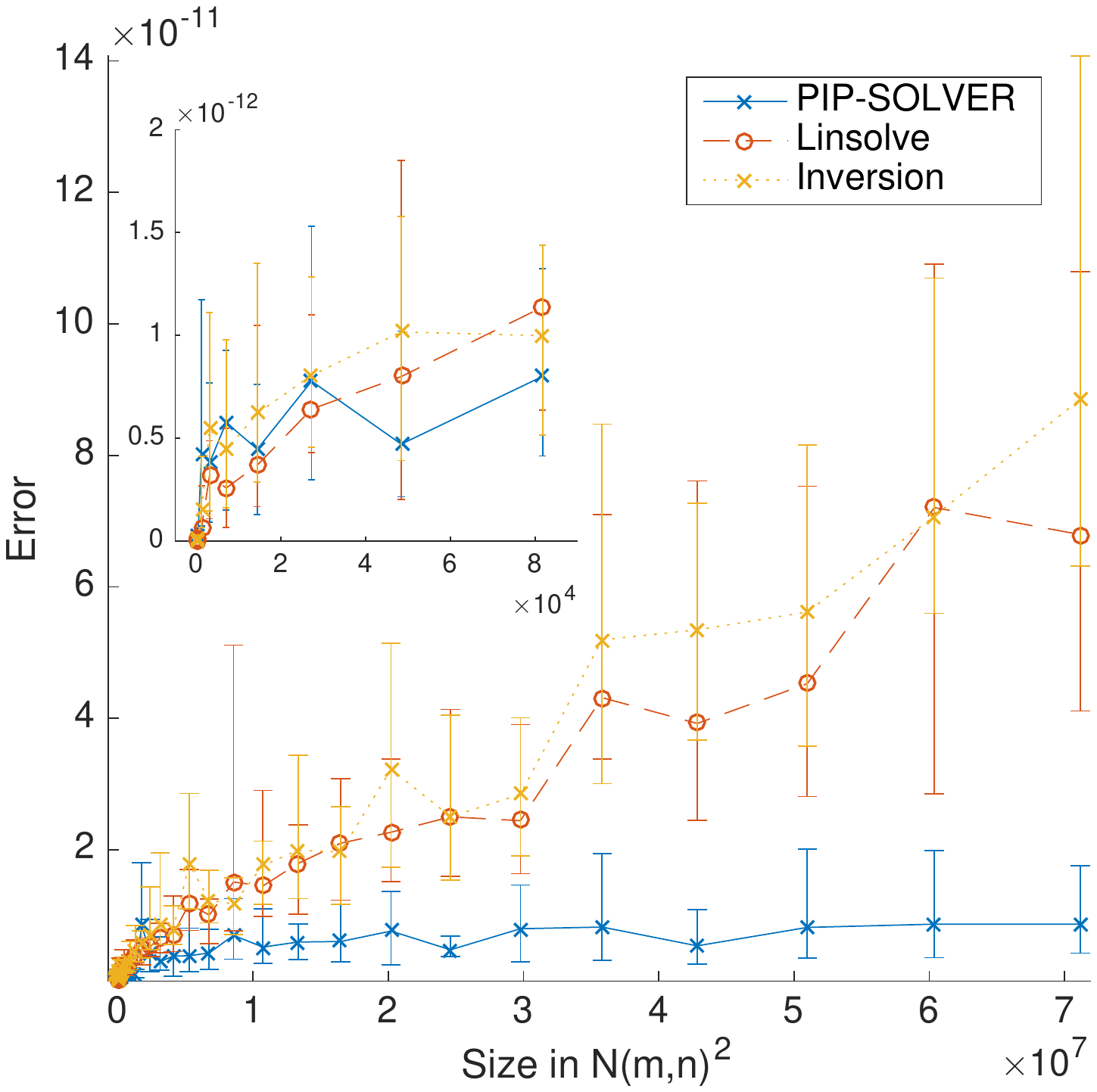}
  \vspace{-3.12cm}
  \caption{Accuracy of the three approaches.}\label{ACC}
 \end{minipage}
\begin{minipage}[t]{0.53\textwidth}
\vspace{-2.38cm}
  \includegraphics[scale=0.4]{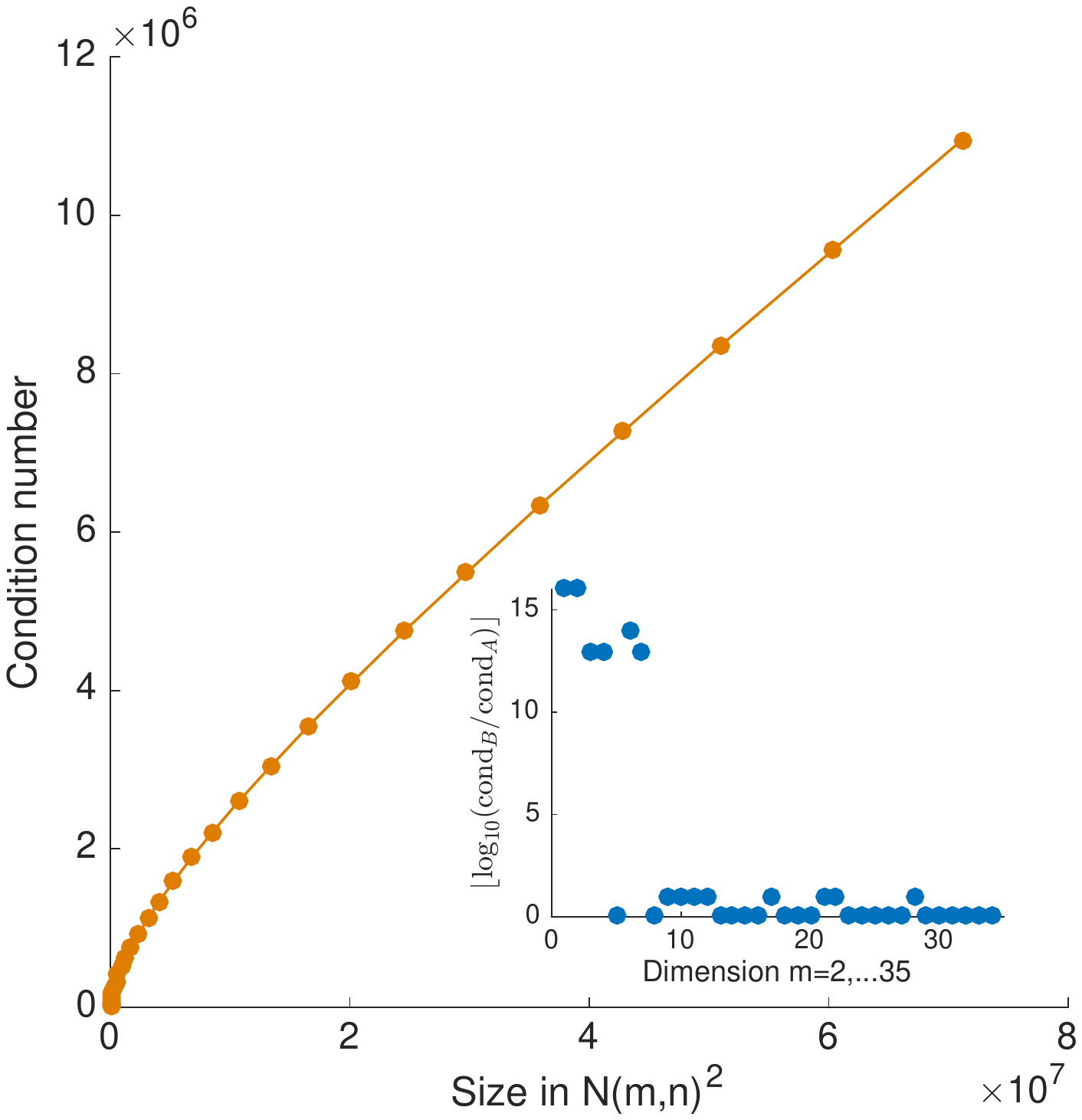}
  \vspace{-2.8cm}
  \caption{Condition of the Vandermonde matrices}\label{condi}
 \end{minipage}
\end{figure}

 
 For given $m,n \in \N$ and function $f: \R^m \lo \R$,
 we compare our approach with the previously used methods. Since there is no general method for finding generic nodes for arbitrary PIP problems, 
 and in order to guarantee comparability of results, we always use our approach for finding the nodes, but then consider three different approaches to solve for the interpolating polynomial on those same nodes: 
 \begin{enumerate}
  \item [i)] Our solution described in the previous section, which we call \emph{PIP-SOLVER}. 
  \item[ii)] Finding generic nodes $P=\{p_1,\dots,p_{N}\}$, $N=N(m,n)$ using our approach, but then solving the PIP by generating the Vandermonde matrix $V_{m,n}(P)$ and using the MATLAB linear solver, which uses an $LU$-decomposition
  to compute
  $C \in \R^N$ with $V_{m,n}(P)C=F$, $F= (f(p_1),\dots,f(p_N))^T$. We call this approach \emph{Linsolve}. 
    \item[iii)] Finding generic nodes $P=\{p_1,\dots,p_{N}\}$, $N=N(m,n)$ using our approach, but then solving the PIP by generating the Vandermonde matrix $V_{m,n}(P)$ and using MATLAB matrix inversion, which is a hybrid algorithm of modern
    matrix inversion approaches, to compute
  $C \in \R^N$ as $C=V_{m,n}(P)^{-1}F$, $F= (f(p_1),\dots,f(p_N))^T$. We call this approach \emph{Inversion}. 
  \end{enumerate}
  
\begin{experiment}
We first compare the accuracy of the three approaches, which also serves to validate our method. 
To do so, we choose uniformly-distributed random numbers $C=(c_0,\dots,c_{N-1}) \in [-1,1]^{N}$, $N=N(\mbox{m,n})$ 
to be the  coefficients of a polynomial $Q_f$ in $m$ variables and degree $\deg(Q_f)=n$. Then we generate a set of generic nodes $P_{m,n}$ and evaluate  
$F= (Q_f(p_1),\dots ,Q_f(p_N))^T$. Afterwards, we compute the solutions $C_k$, $k=1,2,3$ of the PIP with respect to these nodes and the three approaches $(i)$, $(ii)$, $(iii)$ above. 
Finally, we measure the numerical accuracy of each approach $k$ by computing $||C-C_k||_{l^\infty}$, $k=1,2,3$, 
the maximum absolute error in any coefficient. 
 \end{experiment}
 
Figure \ref{ACC} shows the average and min-max span of the errors (over 10 repetitions with different $i.i.d.$~random polynomials; same 10 polynomials for the three approaches) 
for fixed degree $n=3$ and dimensions $m=2,\dots,35$, plotted versus the quadratic size of the problem. The case $n=3$ is of high practical relevance, e.g., when interpolating cubic splines. High degrees $n>5$ are rarely used in practice, since they bear the risk of ringing artifacts. In the case tested here, all methods show high accuracy, which reflects the fact that our generic nodes are well chosen. 
 However, even though all approaches use the same ``good'' generic nodes, we observe a significant gain in 
  accuracy when using the present solver, in particular for large $N$. For small $N$, as shown in the inset for $m=2,\dots,8$, the present approach is as  accurate as the alternatives. The difference in accuracy becomes significant for $m\geq 10$.
We also observe that the error of the present \emph{PIP-SOLVER} approach plateaus and remains almost constant for higher dimensions, while the errors of the alternative approaches seem to increase quadratically with problem size.

Since the numerical accuracy of linear systems solvers and matrix inversion is dictated by the condition number of the Vandermonde matrix, we report these numbers in Figure \ref{condi}. Since the condition number of the Vandermonde matrix only depends on how the generic nodes are chosen, it is the same for all approaches and all random repetitions. It does, however, seem to grow quadratically with problem size, explaining the scaling of the error of the previous approaches.
 In contrast, the error of the \emph{PIP-SOLVER} approach is independent of the condition number of the Vandermonde matrix, because the approach never considers the global Vandermonde matrix. 
 It only works on the small sub-problems of {\em constant size} and {\em constant condition numbers}.
 
Since the condition number of the Vandermonde matrix depends on the set of generic nodes chosen, we also use it to assess the quality of the node sets generated by our approach. We do so by comparing with generic node sets obtained by randomly choosing a sufficiently large subset of nodes from the vertices of a regular Cartesian grid. We denote $\mathrm{cond}_A$ the condition number of the Vandermonde matrix using the generic nodes determined by our PIP-SOLVER and with $\mathrm{cond}_B$
 the maximum condition number over $20$ repetitions of randomly chosen subsets of grid points. Whenever $\mathrm{cond}_A\leq\mathrm{cond}_B$, the resulting log ratios are shown in the inset plot of Figure \ref{condi} 
 as $\log_{10} (\mathrm{cond}_B/\mathrm{cond}_A)$
 versus problem dimension $m=2,\ldots, 35$. However, the difference is less pronounced in higher dimensions, since almost all random node configurations are generic in high dimensions. 
 Though $\mathrm{cond}_A$ might not be the smallest possible condition number, we guarantee that $\mathrm{cond}_A\leq N(m,n)^2$, therefore leading to a well-conditioned Vandermonde matrix $V_{m,n}$, 
 especially for small instances. Note, however, that this is inconsequential for the numerical accuracy of our approach, since the PIP-SOLVER never inverts the Vandermonde matrix.

 \begin{experiment}
We compare the computational runtimes of the three approaches. 
Therefore, we choose uniformly-distributed random function values $F=(f_1,\dots,f_{N})\in [-1,1]^{N}$, $N=N(\mbox{m,n})$ as interpolation targets.
We then measure the time required to generate the generic nodes $P=p_1,\dots,p_{N(m,n)}$ and add the time taken 
by each approach to solve the PIP with respect to $f : \R^m \lo R$ with $f(p_i)=f_i$, $i =1,\dots,N(m,n)$.  
\end{experiment}
\begin{figure}[t!]
\begin{minipage}[t]{0.49\textwidth}
\vspace{-2.5cm}
  \includegraphics[scale=0.4]{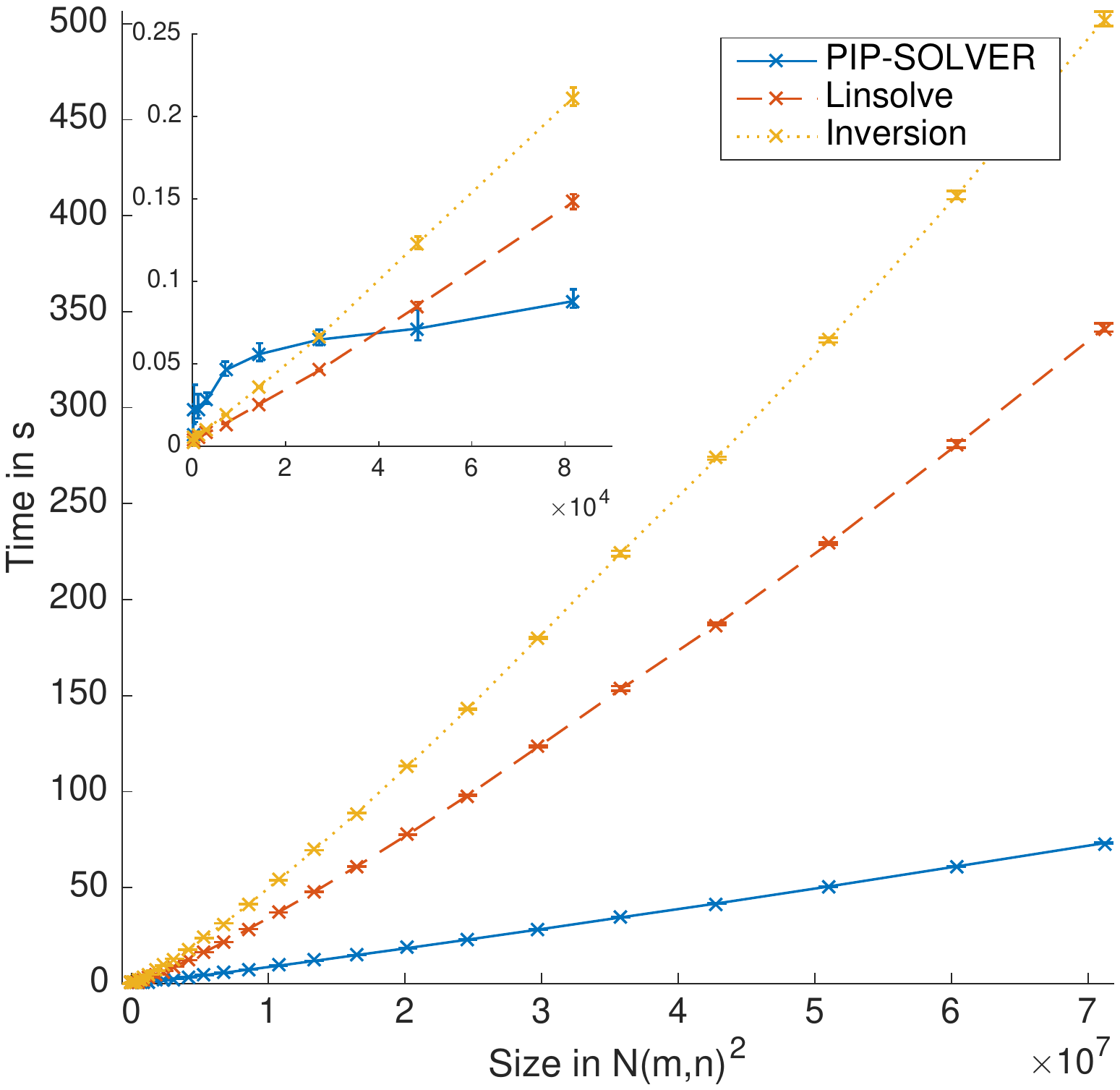}
  \vspace{-3cm}
  \caption{Runtimes for $n=3$, $m=2,\dots,35$  }\label{time1}
 \end{minipage}
\begin{minipage}[t]{0.51\textwidth}
\vspace{-2.5cm}
  \includegraphics[scale=0.4]{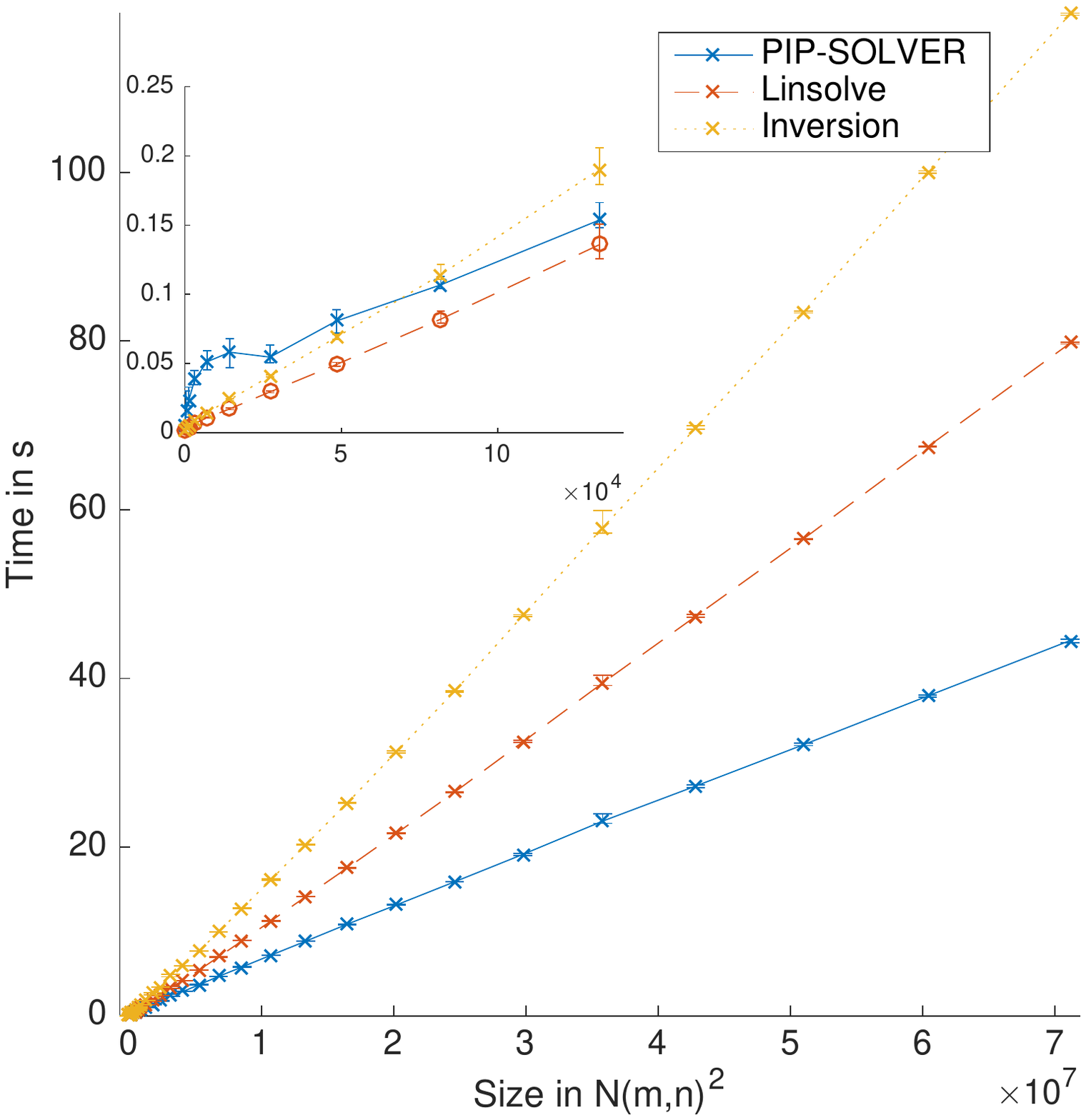}
  \vspace{-3cm}
    \caption{Runtimes for $m=3$, $n=2,\dots,35$  }\label{time2}
 \end{minipage}
\end{figure}

The average and min-max span (over 10 repetitions with different $i.i.d.$~random function values) are shown in Figures \ref{time1} and \ref{time2} versus the square of the problem size. 
Because all repetitions use the same set of generic nodes, the min-max-span error bars are very small. 
The data clearly shown that the \emph{PIP-SOLVER} approach scales significantly better with problem size than the other approaches. 
For small problems, however, the overhead of the PIP-SOLVER may not be amortized and there is a cross-over (for our implementation and benchmark setup at $m=7$ and $n=9$) below which previous methods may be faster. 
This is shown in the inset plots for $m=2,\dots,9$ and $n =2,\dots,10$, respectively.
However, the absolute runtimes are below 0.1 seconds at the point of cross-over, which may not be too relevant in practice. 
Moreover, we expect the cross-over to shift to lower $N$ once a properly optimized implementation of the present approach is available.

The scaling of the computational cost is reported in Table \ref{Tab}, where we fit all curves  with an \emph{R-square} of 1. For fixed dimension and fixed degree we observe that the exponent of our approach is significantly smaller than the exponents of \emph{Linsolve} and \emph{Inversion}, while the pre-factors do not significantly differ.  

However, for fixed degree $n=3$, we observe that the scaling exponent 
$q>2$ for the \emph{PIP-SOLVER}, although we have proven a quadratic time behavior in Theorem \ref{mainT}. 
This is likely due to the fact that our prototype MATLAB implementation is not optimal. We expect that an optimized implementation of the \emph{PIP-SOLVER} in a natively compiled programming 
language is going to reduce the exponent. However, the results here confirm that 
the PIP-SOLVER scales better than any previous approach.

In addition to having a lower time complexity, PIP-SOLVER also requires less memory and has a lower space complexity. Indeed, the PIP-SOLVER requires only $\Oc(mN(m,n))$ storage, 
whereas both previous approaches require $\Oc(N(m,n)^2)$ storage to hold the Vandermonde matrix. Due to this lower space complexity, we could solve the PIP for large instances, i.e., $N(m,n)\geq 10^5$ 
within reasonable time (hours), while both previous approaches failed to solve such large problems due to insufficient memory on the computer used for the experiment.

\section{Applications}\label{APP} We would like to highlight some potential applications of the present PIP-SOLVER in scientific computing and computational science. However, this list is, by no means exhaustive, as PIPs are a fundamental component of many numerical methods. The following applications may not be obvious, though:  

{\bf A1)} Given $m,n\in \N$, $m\geq 1$ and a function $f : \R^m \lo \R$. It is classical in numerical analysis to determine the integral $\int_{\Omega}f \,\mathrm{d}\Omega$, $\Omega \subseteq \R^m$  
and the  derivative  $\partial_vf(x) $ of $f$ in direction $v \in \R^m$ at $x \in \R^m$. Solving the PIP, these desired quantities can easily be computed for the interpolation polynomial $Q$ of $f$ evaluated at some collocation points, converging to 
to the quantities of $f$ with increasing degree $\deg(Q_f)=n$. A comparison with other approaches from numerical analysis would therefore be interesting. 

\begin{table}[t!]
 {\renewcommand{\arraystretch}{1.2}
\begin{center}
{\footnotesize
\begin{tabular}[t]{| l | l | l | l |} 
\hline 
Algorithm & $m,n$ intervals &  Pre-factor $p$ & Exponent $q$\\ 
 \hline

 \emph{PIP-SOLVER} &$m =2,\dots,35 $, $n=3$&          $p = 2.7624\cdot 10^{-7}$ &  $q=2.1427  $ \\
 \hline
  \emph{Linsolve} & $m =2,\dots,35 $, $n=3$&         $p=2.4846\cdot 10^{-7}$ & $q=2.3262 $ \\ 
 \hline
  \emph{Inversion} &$m =2,\dots,35 $, $n=3$ & $p=3.4421\cdot 10^{-7}$ &   $q=2.3322 $ \\ 
 \hline
 &  &         &  \\
 \hline
  \emph{PIP-SOLVER} & $m =3$, $n= 2,\dots,35$&         $p = 2.1652\cdot 10^{-6}$ &  $q=1.8588  $ \\
 \hline
  \emph{Linsolve} & $m=3$, $n =2,\dots,35 $&     $p=7.2797\cdot 10^{-7}$ & $q=2.0467 $ \\ 
 \hline
  \emph{Inversion} & $m=3$,  $n =2,\dots,35 $&       $p=8.6115\cdot 10^{-7}$ &   $q=2.0711 $ \\ 
 \hline
\end{tabular} 
}
\vspace{0.2cm}
\end{center}
\caption{Scaling of the computational cost by fitting the polynomial model $px^q$.}}\label{Tab}
\end{table}

{\bf A2)} A maybe surprising application is found in cryptography. There, the PIP is used to "share a secret" by choosing a random polynomial $Q \in \Z[x]$ in dimension $m=1$. Knowing the values of $Q$ at $n+1$ different nodes 
enables one to determine \mbox{$Q(0) \!\!\! \mod p$} for some 
large prime number $p \in \N$. However, knowing only $n$ "keys" prevents one for opening the "door".
Certainly, this method can be generalized to arbitrary dimensions and therefore our approach applies to this type of cryptography, see also \cite{Shamir}. In particular, 
the improved numerical accuracy of PIP-SOLVER would prevent the reconstructed message from being corrupted by numerical noise.

{\bf A3)} \emph{Gradient descent} over multivariate functions is often used to solve non-convex optimization problems, where $\Omega \subseteq \mathbb{R}^m$ models the space of possible solutions for a given problem 
and $f: \R^m \lo \R$ is interpreted as an objective function. Thus, one wants to minimize $f$ on $\Omega$.   
Often, the function $f$ is not explicitly known $\forall x \in \Omega$, but can be evaluated point-wise.
If \mbox{$m \gg 1$}, an interpolation of $f$ on $\Omega$
is often not possible or not considered so far. 
  
Our approach allows interpolating $f$ even for \mbox{$m \gg 1$}.  Then, we can sample through the convex cones of the interpolation polynomial $Q_f$ and apply the classical \emph{Newton-Raphson method} to find the corresponding local optima. 
Alternatively, one might be able to solve the global flow equation $\dot x(t)=-\nabla f(x(t))$, $x(0)=x_0$. In the generic situation, i.e., for almost all $x_0$ we have that $x(+\infty)$ is a local minimum. 
Thus, by sampling over the initial conditions, the global optimum of the interpolation polynomial might be found, approximating the global optimum of $f$.  
We expect that the accuracy and runtime performance of the present PIP-SOLVER can be used to improve gradient-based optimization algorithms. 

\section{Discussion and Conclusions}

In closing, we summarize our results in the literature context, sketch possible generalizations, and list remaining open problems. 

\label{Conc}
\subsection{Summary} An excellent survey  of approaches and contributions to the PIP is given in \cite{Gasca2000}. The idea of adapting B\'{e}zout's Theorem in order to decompose the PIP w.r.t.~$m,n \in \N$ into 
two sub-problems of dimension and degree \mbox{$(m-1,n)$}, \mbox{$(m,n-1)$}, respectively, has already been 
mentioned in \cite{Guenther,2000}. In \cite{Gasca}, the cases $m=2,3$ were treated explicitly and a generalization to arbitrary dimension was sketched. Due to these approaches and others \cite{Gasca2000},
some characterizations of generic nodes in arbitrary dimensions could be given.   
However, the problem of computing the interpolation polynomial $Q$ efficiently and accurately remained unsolved. Indeed, all previous decomposition approaches were limited to bounded dimension and to nodes on 
pre-defined grids or meshes \cite{Bos,Erb,Gasca2000,Chung}. 

To our knowledge, this work is the first to present a decomposition of the general PIP into sub-problems that can accurately and robustly be solved without any pre-defined grid or mesh.
We also believe that this work is the first to show how the numerical or analytical solutions of the sub-problems can be combined to a full solution of the original PIP.
As verified in Section \ref{EX}, our algorithm PIP-SOLVER provides  
significant improvements in both accuracy and runtime when compared  
with other general approaches based on linear systems solvers or matrix inversion. Importantly, the improvements are not in the pre-factor, but in how the accuracy and performance scale with problem size. 
This makes it possible to address problems of unprecedented size at constant numerical accuracy. We expect that the runtime  is further reduced by an optimized implementation of the PIP-SOLVER in a 
compiled programming language. 
Thus, as long as the generic nodes $P$ can be freely chosen in $\R^m$, we have presented an efficiently computable procedure for solving general PIPs.  
The runtime, memory requirement, and numerical accuracy of the resulting algorithm improve over the previous state of the art. 
In particular, we overcome the theoretical lower bound on matrix inversion of $\Oc(N^2\log(N))$, $N=N(m,n)$, \cite{tveit}. 

\subsection{Generalizations}\label{GEN}
We foresee the possibility of generalizing the characterization of generic sets from dimension 2 to arbitrary dimensions by 
considering intersections of algebraic varieties instead of algebraic curves. This could lead to finding a similar estimate for the number of intersection points as given by B\'{e}zout's Theorem.
Furthermore, the decomposition in Theorems \ref{GN} and \ref{SV} can possibly be generalized by replacing hyperplanes $H_\ee$ by arbitrary hyper-surfaces of $1\leq \deg(H_\ee)\leq n$. 

Algorithmically, a hybrid approach could be implemented that combines the present decomposition scheme with directly inverting the Vandermonde matrix for small sub-problems, i.e., for $N \leq 5\ldots 8 \cdot 10^4$, corresponding to the crossover points seen in Figures \ref{time1}, \ref{time2}.
Moreover, the PIP can also be understood with respect to other basis functions, such as \emph{Chebyshev polynomials}  or \emph{Fourier bases}. The present approach should also work there, possibly enabling the implementation of 
another \emph{Fast Multivariate Fourier Transform} that improves the runtime behavior of current approaches. 
Moreover, over-fitting problems could be addressed by applying the present approach to \emph{spline interpolation} and \emph{Hermite interpolation}.

From a software-engineering viewpoint, a distributed-memory parallelized version of the algorithm could be implemented in order to further reduce runtimes and alleviate memory limitations. This is possible since our approach yields a decomposition into independent sub-problems that can be processed in parallel, using inter-process communication to ensure correct decomposition of the problem and synthesis of the final solution.

\subsection{Remaining Problems}
Several remaining problems can be identified pertaining to the details of how the function $f$ is given. 
Frequently, for example, the function $f: \R^m\lo \R$, $m \in \N$ is not known for every $x \in \R^m$, but only on a given set of collocation points $\Pc \subseteq \R^m$. 
The problem remains to adapt the decomposition of the PIP in a way that allows choosing generic nodes $P \subseteq \Pc$ for a given, fixed $\Pc$. 
In cases where the points $\Pc$ sit on certain lattices or grids, this is easily achieved by choosing the hyperplanes $H_\ee \subseteq \R^m$ 
such that they include the corresponding sub-lattice or sub-grid.  
If the nodes $\Pc$ are arbitrarily distributed, finding the right PIP decomposition remains an open problem. 
The results of Section \ref{2D} provide some freedom to adapt the decomposition by choosing admissible solutions in dimension 2 in cases where a  decomposition into $1$-dimensional sub-problems is not feasible.
In addition, generalizing the approach from hyperplanes to hyper-surface, as sketched in Section \ref{GEN}, might provide sufficiently rich configurations for which one can guarantee that
$\Pc \cap H_\ee $
is sufficient large to generate generic node sets for the corresponding sub-problem. 

With respect to the numerical accuracy of the solver, it remains an open problem to find the optimal way of decomposing the PIP with respect to the requirements of Theorems \ref{GN},\ref{SV} such that the condition numbers of the resulting sub-problems are the smallest possible, and the numerical accuracy is maximized. Theoretical proofs of optimality would be most useful, also in designing the optimal algorithm.

Taken together, the present work highlights a number of follow-up problems that may eventually lead to a full and complete understanding of general PIPs.


\section*{Acknowledgments} We thank Rajesh Ramaswamy for inspiring discussions that have motivated us to study this problem.


\end{document}